\documentclass{article}
\usepackage{amsmath,amssymb,amsthm}
\usepackage{enumerate}
\usepackage{mathdots}

\newtheorem{theorem}{Theorem}[section]
\newtheorem{lemma}[theorem]{Lemma}
\newtheorem{proposition}[theorem]{Proposition}
\newtheorem{corollary}[theorem]{Corollary}
\newtheorem{definition}[theorem]{Definition}
\theoremstyle{definition}
\newtheorem{remark}[theorem]{Remark}
\newtheorem{example}[theorem]{Example}

\newcommand{\Z}{\mathbb{Z}}

\newcommand{\im}{\operatorname{Im}}
\renewcommand{\ker}{\operatorname{Ker}}
\newcommand{\id}{\operatorname{id}}
\newcommand{\Id}{\operatorname{Id}}

\newcommand{\aut}{\operatorname{Aut}}

\newcommand{\soc}{\operatorname{Soc}}
\newcommand{\ord}[1]{\vert #1\vert}

\begin{document}

\title{Extensions, matched products, and simple braces}
\author{David Bachiller}
\date{}
\maketitle

\begin{abstract}
We show how to construct all the extensions of left braces by ideals 
with trivial structure. This is useful to find new examples of left braces. 
But, to do so, we must know the basic blocks for extensions: the left 
braces with no ideals except
the trivial and the total ideal, called simple left braces. 
In this article, we present the first non-trivial examples of finite simple
left braces.
To explicitly construct them, we define the matched product of two left braces, 
which is a useful method to recover a finite left brace from its Sylow subgroups. 
\end{abstract}

{\bf Keywords:} Left braces, extensions, simple braces, matched product, Yang-Baxter equation.

{\bf MSC:} 20F16, 20E22, 20F28, 16T25.

\section{Introduction}

Since its appearance in theoretical physics, the Yang-Baxter equation has motivated
the research of many algebraic structures, like Hopf algebras and quantum groups (see \cite{Kassel}). 
Since the construction of new solutions of this equation is widely open, 
Drinfeld in \cite{Drinfeld} suggested to study some subclasses of solutions with additional properties. 
One class of solutions that can be studied with techniques from group theory
has received a lot of attention recently, the non-degenerate involutive 
set-theoretic solutions (see \cite{CJO}, and the references there). 
Besides its relation with theoretical physics, this class of solutions is important 
for its connections with many branches of algebra, for instance:
semigroups of $I$-type and Bieberbach groups
\cite{GVdB}, bijective $1$-cocycles \cite{ESS}, radical rings
\cite{Rump}, triply factorized groups \cite{sysak}, construction of
semisimple minimal triangular Hopf algebras \cite{EtingofGelaki},
racks, quandles and pointed Hopf algebras \cite{AG},
regular subgroups of the holomorf and Hopf-Galois extensions
\cite{CR,FCC}, and groups of central type \cite{NirBenDavid,BDG}.

Rump in \cite{Rump} introduced 
a new algebraic structure called brace to study this class of solutions.
A left brace is a set $B$ with two binary operations, a sum $+$ and a product $\cdot$, such that
$(B,+)$ is an abelian group, $(B,\cdot)$ is a group, and for any $a,b,c\in B$,
$$
a\cdot (b+c)+a=a\cdot b+a\cdot c.
$$
To any non-degenerate involutive set-theoretic solution $(X,r)$, we can associate a left brace
$\mathcal{G}(X,r)$, called the permutation group of the solution. 
Conversely, for any left brace $B$, there exists a solution $(X,r)$ of this class 
such that $\mathcal{G}(X,r)\cong B$. 
But moreover, using the recent results of \cite{BCJ}, given a left brace $B$, we can recover 
from $B$ all the solutions $(X,r)$ such that $\mathcal{G}(X,r)\cong B$ as left braces. 
Hence for any new left brace that we obtain, we are able to construct an infinite number 
of new solutions of the Yang-Baxter equation. And, theoretically, 
a classification of the left braces implies the possibility to construct all the
non-degenerate involutive set-theoretic solutions of the Yang-Baxter equation. 

To begin to study the classification problem, we need new methods to construct left braces. 
An interesting method might be extensions of left braces: given two left braces $I$ and $H$, determine all the 
possible left braces $B$ such that $0\to I\to B\to H\to 0$ is exact; i.e. such that there exists 
an ideal $J$ of $B$ isomorphic to $I$ satisfying $B/J\cong H$. 
Until now, there has been only one construction of this type, given for the first time 
in \cite[Corollary D]{NirBenDavid} and translated into brace language in 
\cite[Theorem 2.1]{B}, 
but it only works when $I$ is contained in an specific ideal of $B$, the socle of $B$. 
In this article, we develop the theory of extensions in the case that $I$ is a 
trivial brace; i.e. $y_1\cdot y_2=y_1+y_2$ for any $y_1,y_2\in I$. The reason to restrict 
to this case is analogous to the reason why most books only define extensions of groups 
with abelian kernel: because, to construct the extension $B$, in the non-trivial case there are 
some technical troubles when we define the actions of $H$ on $I$. 
Note that the socle of $B$ is a trivial brace, and that there are left braces with trivial socle,
 so the results of our article both contain and generalize
\cite[Corollary D]{NirBenDavid}.

Besides methods to construct new left braces from smaller ones, we need to
find the initial blocks to begin the construction.  
From the point of view of extensions, this initial blocks are the simple left braces: 
left braces with no ideals except the trivial and the total ideals. 
Until now, the only known examples of finite simple left braces were braces 
with additive group equal to $\Z/(p)$ and with 
product defined by $x\cdot y:=x+y$, for any $x,y\in \Z/(p)$. 
In this article, we present the first non-trivial examples of 
finite simple left braces. 

To find simple left braces, we develop another technique to construct new left braces 
from two given braces, called matched product of left braces: 
given $L_1$ and $L_2$ two left braces, we give sufficient conditions 
to construct a left brace $B$ such that $L_1$ and $L_2$ are 
left ideals inside $B$, and $(B,+)\cong (L_1,+)\oplus(L_2,+)$. 
Moreover, any left brace with two left ideals 
satisfying these properties is of this form. 
 This is analogous to the matched product of two subgroups introduced in group theory, 
 as it was defined in \cite{Mackey,Takeuchi2}.
  It is possible to define matched products of many algebraic structures and, in particular, 
  matched products of bialgebras are a key step to construct 
the Drinfeld double of a Hopf algebra, a construction related to the Yang-Baxter equation; 
see for example \cite[Chapter~IX]{Kassel}. 
  In particular, in the finite case, the Sylow subgroups 
of $(B,+)$ are left ideals with trivial intersection and whose sum is all $B$, so we can recover $B$ from 
these left ideals of order a power of a prime, under the conditions of Theorem \ref{product}. 

This technique is very useful because in \cite[Corollary after Proposition 8]{Rump} it is proved that 
a left brace with order $p^n$, where $p$ is a prime number and $n\geq 2$, is never a simple left brace. 
Thus any other finite simple left brace has at least Sylow subgroups for two different primes,
and we can put them together with a matched product. We make explicit this idea in the 
two last sections of this article.

The paper is organized as follows: in Section 2, we give the necessary preliminary definitions and 
results about left braces to make the paper self-contained. 
Next, in Section 3, given a trivial brace $I$ and a left brace $H$, we give a method to 
construct all the extension $0\to I\to B\to H\to 1$ in terms of $2$-cocycles, analogous to 
the method in group theory.
In Section 4, we define the notion of matched product of braces, which in particular provides us 
with a method to glue together two Sylow subgroups for different primes. 
With this idea in mind, 
in Section 5 we recall a family of left braces with order a power of a prime, 
first appearing in \cite{Hegedus}. The special property of this family is that 
all its braces have zero socle. 
In Section 6, using the left braces from Section 5 and the matched product 
of Section 4, we give sufficient conditions to construct examples of simple left brace of 
order $p_1\cdot p_2^{n+1}$, where $p_1$ and $p_2$ are different prime numbers such that 
$p_2\mid (p_1-1)$, and $n\geq 1$. The Sylow $p_2$-subgroup of all these examples is always one of the 
braces of Section 5.
Note that all these examples have order of the form $p^a\cdot q^b$, and that we even obtain examples of odd order; 
so although finite left braces can be defined in group theoretical terms, their algebraic structure seems to be very 
different from that of finite groups. 
Finally, in Section 7, we make use of the theory of the previous section to give the 
first concrete examples of non-trivial simple left braces, by explicitly constructing 
all the needed elements of the theorems of Section 6. We obtain the existence of a simple left brace 
of order $p_1\cdot p_2^{k(p_1-1)+1}$ for any positive integer $k$, and any pair of different primes 
$p_1$ and $p_2$ with $p_2\mid (p_1-1)$.

\section{Initial results about left braces}

\begin{definition}
A left brace is a set $B$ with two binary operations, a sum $+$ and a product $\cdot$, such that
$(B,+)$ is an abelian group, $(B,\cdot)$ is a group, and for any $a,b,c\in B$,
$$
a\cdot (b+c)+a=a\cdot b+a\cdot c.
$$
A right brace is defined similarly, but changing the last property by $(b+c)\cdot a+a=b\cdot a+c\cdot a.$
A left and right brace is called a two-sided brace.
\end{definition}

\begin{remark}
In a left brace, the additive and the multiplicative neutral elements coincide, since
$$
1\cdot 0+1=1\cdot(0+0)+1=1\cdot 0+1\cdot 0,
$$
implying $1=1\cdot 0=0$.
\end{remark}

We now recall the definition of lambda maps of a left brace, which play a major role in 
next sections. 

\begin{lemma}(\cite[Lemma 1]{CJO})
Let $B$ be a left brace. For any $a\in B$, define a map $\lambda_a: B\to B$
given by $\lambda_a(b)=a\cdot b-a$, for any $b\in B$. Then, $\lambda_a$ 
is an automorphism of $(B,+)$ and the map $\lambda: (B,\cdot)\to\aut(B,+)$,
$a\mapsto \lambda_a$,
is a group homomorphism.
\end{lemma}

We recall next some concepts in brace theory.

\begin{definition}
A \emph{sub-brace} of $B$ is a subset closed by sum and multiplication of its elements. 
A \emph{left ideal} is a subgroup $L$ of $(B,\cdot)$ such that $\lambda_b(x)\in L$ for any $b\in B$ and $x\in L$.
An \emph{ideal} is a normal subgroup $I$ of $(B,\cdot)$ such that $\lambda_b(x)\in I$ for any $b\in B$ and $x\in I$.

A \emph{simple left brace} is a non-zero left brace $B$ such that the only ideals of $B$ are $0$ and $B$.
\end{definition}

\begin{definition}
We say that a left brace is trivial if $a\cdot b=a+b$ for any $a$ and $b$
elements of that brace.
\end{definition}

Given an abelian group, we can define a structure of left
brace over it through the lambda map, as it is explained in the next lemma.
Sometimes it is easier to prove that some set has a brace structure using this lemma. 

\begin{lemma}\label{defLambda}
Let $(A,+)$ be an abelian group. Let $\lambda:A\to \aut(A,+)$, $a\mapsto \lambda_a$, be a map such that
$\lambda_a\circ\lambda_b=\lambda_{a+\lambda_a(b)}$, for every $a,b\in A$.
Then, $A$ with the product $a\cdot b:=a+\lambda_a(b)$ is a left brace.

Two of these braces, defined over the same abelian group $A$ through lambda maps $\lambda$ and $\lambda'$ resp., are isomorphic if and only if
$\lambda'_{F(a)}=F\lambda_a F^{-1}$ for some automorphism $F$ of $A$.
\end{lemma}

\begin{proof}
It is easy to check that these conditions are satisfied in any left brace.
Conversely, given a map with these properties, define the product $a\cdot b:=a+\lambda_a(b)$.
We shall check
all the properties of a left brace.
Observe first that the condition $\lambda_a\circ\lambda_b=\lambda_{a+\lambda_a(b)}$ with $a=b=0$ implies that $\lambda_0=\id$.
Then $0$ is the neutral element of the product, because 
$a\cdot 0=a+\lambda_a(0)=a+0=a,$ and $0\cdot a=0+\lambda_0(a)=a.$
Moreover, note that the condition $\lambda_a\circ\lambda_b=\lambda_{a+\lambda_a(b)}$
with $b=-\lambda^{-1}_a(a)$ implies that $\lambda_a^{-1}=\lambda_{-\lambda_a^{-1}(a)}$.
Then any element has an inverse element, which is 
$a^{-1}:=-\lambda^{-1}_a(a)$, since $a\cdot (-\lambda^{-1}_a(a))=a+\lambda_a(-\lambda_a^{-1}(a))=a-a=0,$ and 
$(-\lambda^{-1}_a(a))\cdot a=-\lambda_a^{-1}(a)+\lambda_{-\lambda_a^{-1}(a)}(a)=-\lambda_a^{-1}(a)+\lambda^{-1}_{a}(a)=0.$

To check the associativity property, on one side,
$$a\cdot (b\cdot c)=a+\lambda_a(b+\lambda_b(c))=a+\lambda_a(b)+\lambda_a(\lambda_b(c)),$$
and, on the other,
 $$(a\cdot b)\cdot c=(a+\lambda_a(b))\cdot c=a+\lambda_a(b)+\lambda_{a+\lambda_a(b)}(c)=a+\lambda_a(b)+(\lambda_a\circ\lambda_b)(c).$$
Finally, the brace property is satisfied:
$$a\cdot(b+c)+a=a+\lambda_a(b+c)+a=a+\lambda_a(b)+\lambda_a(c)+a=a\cdot b+a\cdot c.$$
\end{proof}

\section{Extensions of left braces}\label{extensions}

Now we define a notion of extensions of left braces, mimicking the extensions of group theory. 
Take $B$ a left brace, $I$ an ideal, and $H=B/I$.  
Since the natural projection $\pi: B\to H$ is surjective, we can choose a map $s: H\to B$ such that 
$\pi\circ s=\id$ and $s(0)=0$ (i.e. we choose a set of coset representatives of $(I,+)$ in $(B,+)$). 
Then any element of $B$ can be uniquely written as $s(h)+y$, where $h\in H$ and $y\in I$. The sum of two elements of 
this form is determined by $(s(h_1)+y_1)+(s(h_2)+y_2)=s(h_1+h_2)+(s(h_1)+s(h_2)-s(h_1+h_2)+y_1+y_2)$. 
Note that $s(h_1)+s(h_2)-s(h_1+h_2)\in I$, so we define a map $\beta: H\times H\to I$, 
$\beta(h_1,h_2)=s(h_1)+s(h_2)-s(h_1+h_2).$

On the other hand, the product of two elements is determined by 
\begin{align*}
(s(h_1)+y_1) (s(h_2)+y_2)&=s(h_1)\cdot\lambda^{-1}_{s(h_1)}(y_1)\cdot s(h_2)\cdot\lambda^{-1}_{s(h_2)}(y_2)\\
&=s(h_1)\cdot s(h_2)\cdot s(h_2)^{-1}\cdot\lambda^{-1}_{s(h_1)}(y_1)\cdot s(h_2)\cdot\lambda^{-1}_{s(h_2)}(y_2)\\
&=s(h_1\cdot h_2)\cdot[s(h_1\cdot h_2)^{-1}\cdot s(h_1)\cdot s(h_2)\\
&\quad\cdot s(h_2)^{-1}\cdot\lambda^{-1}_{s(h_1)}(y_1)\cdot s(h_2)\cdot\lambda^{-1}_{s(h_2)}(y_2)]\\
&=s(h_1\cdot h_2)+\lambda_{s(h_1\cdot h_2)}[s(h_1\cdot h_2)^{-1}\cdot s(h_1)\cdot s(h_2)\\
&\quad\cdot s(h_2)^{-1}\cdot\lambda^{-1}_{s(h_1)}(y_1)\cdot s(h_2)\cdot\lambda^{-1}_{s(h_2)}(y_2)].
\end{align*}
Observe that $s(h_1\cdot h_2)^{-1}\cdot s(h_1)\cdot s(h_2)\in I$, so we define a map 
$\tau: H\times H\to I$ by $\tau(h_1,h_2)=s(h_1\cdot h_2)^{-1}\cdot s(h_1)\cdot s(h_2)$. 

Since $I$ is a left ideal, we can restrict the lambda morphism $\lambda: (B,\cdot)\to\aut(B,+)$ to a morphism 
$\lambda_{I}: (B,\cdot)\to\aut(I,+)$.
If we want that $\lambda_I$ determines an action $B/I\to\aut(I,+)$, we must 
assume that $I\subseteq \ker(\lambda_I)$. This means that 
$y_1\cdot y_2=y_1+\lambda_{y_1}(y_2)=y_1+y_2$ for any $y_1, y_2\in I$; i.e. $I$ must be a trivial brace.

Hence, to simplify things, we will assume from now on that $I$ is a trivial brace. 
As a consequence, $(I,\cdot)$ becomes abelian, and then 
$\sigma_{x}(y):=s(x)^{-1}\cdot y\cdot s(x)$ becomes a right action, since 
\begin{align*}
\sigma_{x_1}(\sigma_{x_2}(y))&=s(x_1)^{-1}(s(x_2)^{-1}ys(x_2))s(x_1)\\
&=\tau(x_2,x_1)^{-1}s(x_2x_1)^{-1}ys(x_2x_1)\tau(x_2,x_1)\\
&=s(x_2x_1)^{-1}ys(x_2x_1)\tau(x_2,x_1)^{-1}\tau(x_2,x_1)\\
&=s(x_2x_1)^{-1}ys(x_2x_1)=\sigma_{x_2x_1}(y),
\end{align*}
 using in the third equality
that $I$ is abelian. 
Moreover, $\lambda_{s(h_1)}\circ\lambda_{s(h_2)}=\lambda_{s(h_1)\cdot s(h_2)}=
\lambda_{s(h_1\cdot h_2)\tau(h_1,h_2)}=\lambda_{s(h_1\cdot h_2)}$ because 
$ \tau(h_1,h_2)\in I$, so $\nu_h:=\lambda_{s(h)}$ becomes a left action.

The fact that the sum must define a structure of abelian group, and the product, a structure of group,
 implies some properties for 
$\tau$ and $\beta$, as in the following definition.

\begin{definition}
Let $G$ be a group, let $M$ be an abelian group, and let $\sigma: G\to\aut(M)$ be a right action. 
We say that a map $\tau: G\times G\to M$ is a 2-cocycle with respect to $\sigma$ 
if it satisfies 
$\sigma_{g_3}(\tau(g_1,g_2))+\tau(g_1\cdot g_2,g_3)=\tau(g_1,g_2\cdot g_3)+\tau(g_2,g_3)$,
and $\tau(g,1)=\tau(1,g)=0$.

Let $A$ and $M$ be two abelian groups. We say that a map 
$\beta: A\times A\to M$ is an additive $2$-cocycle if it satisfies
$\beta(a_2,a_3)-\beta(a_1+ a_2,a_3)+\beta(a_1,a_2+ a_3)-\beta(a_1,a_2)=0,$ 
$\beta(a_1,a_2)=\beta(a_2,a_1),$ and 
$\beta(a,0)=\beta(0,a)=0$.
\end{definition}

Now we check that our $\tau(h_1,h_2):=s(h_1\cdot h_2)^{-1}\cdot s(h_1)\cdot s(h_2)$, and our 
$\beta(h_1,h_2):=s(h_1)+s(h_2)-s(h_1+h_2)$ are $2$-cocycles. 
Since the sum in $B$ is associative, $s(h_1)+(s(h_2)+s(h_3))=(s(h_1)+s(h_2))+s(h_3)$. Using that
$s(h+h')+\beta(h,h')=s(h)+s(h')$, the left hand side of this equality becomes
\begin{align*}
s(h_1)+(s(h_2)+s(h_3))&=s(h_1)+(s(h_2+h_3)+\beta(h_2,h_3))\\
&=s(h_1+h_2+h_3)+\beta(h_1,h_2+h_3)+\beta(h_2,h_3),
\end{align*}
and the right hand side becomes
\begin{align*}
(s(h_1)+s(h_2))+s(h_3)&=(s(h_1+h_2)+\beta(h_1,h_2))+s(h_3)\\
&=s(h_1+h_2)+s(h_3)+\beta(h_1,h_2)\\
&=s(h_1+h_2+h_3)+\beta(h_1+h_2,h_3)+\beta(h_1,h_2),
\end{align*}
obtaining the $2$-cocycle relation for $\beta$.
The fact that the sum is commutative implies $\beta(h_1,h_2)=\beta(h_2,h_1)$, and $s(0)=0$ implies that 
$\beta(0,h)=\beta(h,0)=0$.

Next, since the product in $B$ is associative, $s(h_1)\cdot(s(h_2)\cdot s(h_3))=(s(h_1)\cdot s(h_2))\cdot s(h_3)$. Using that
$s(h\cdot h')\cdot \tau(h,h')=s(h)\cdot s(h')$, the left-hand side of this equality becomes
\begin{align*}
s(h_1)\cdot(s(h_2)\cdot s(h_3))&=s(h_1)\cdot (s(h_2\cdot h_3)\cdot\tau(h_2,h_3))\\
&=s(h_1\cdot h_2\cdot h_3)\cdot\tau(h_1,h_2\cdot h_3)\cdot\tau(h_2,h_3),
\end{align*}
and the right-hand side of the equality is equivalent to
\begin{align*}
(s(h_1)\cdot s(h_2))\cdot s(h_3)&=(s(h_1\cdot h_2)\cdot\tau(h_1,h_2))\cdot s(h_3)\\
&=s(h_1\cdot h_2)\cdot s(h_3)(s(h_3)^{-1}\beta(h_1,h_2)s(h_3))\\
&=s(h_1\cdot h_2\cdot h_3)\tau(h_1\cdot h_2,h_3)\sigma_{h_3}(\tau(h_1,h_2)),
\end{align*}
obtaining thus the $2$-cocycle condition for $\tau$.
The fact that $\tau(1,h)=\tau(h,1)=0$ comes from the fact that $s(0)=0$, and that $1=0$ in a left brace.

Moreover, the product and the sum must be related by the left brace property, and this implies a property 
relating $\tau$ and $\beta$. 

\begin{definition}\label{2cocycle}
Let $H$ be a left brace, and let $I$ be an abelian group. We think $I$ as a  
trivial brace. Let $\sigma: (H,\cdot) \to\aut(I,+)$ be a right action, and let
$\nu: (H,\cdot)\to\aut(I,+)$ be a left action. We say that $(H,I,\sigma,\nu,\tau,\beta)$ is an extension 
if $\tau: (H,\cdot)\times (H,\cdot)\to (I,+)$ is a $2$-cocycle with respect to $\sigma$, 
$\beta: (H,+)\times (H,+)\to (I,+)$ is an additive $2$-cocycle, the two actions are related by
\begin{align}\label{propact}
\nu_{h_1+h_2}(\sigma_{h_1+h_2}(y))+y=\nu_{h_1}(\sigma_{h_1}(y))+\nu_{h_2}(\sigma_{h_2}(y)),
\end{align}
and $\tau$ and $\beta$ are related by
\begin{align}\label{propcocycle}
\nu_{h_1(h_2+h_3)}(\tau(h_1,h_2+h_3))+\beta(h_1(h_2+h_3),h_1)+\nu_{h_1}(\beta(h_2,h_3))\nonumber\\
=\nu_{h_1 h_2}(\tau(h_1,h_2))+\nu_{h_1h_3}(\tau(h_1,h_3))+\beta(h_1h_2,h_1h_3),
\end{align}
for any $h_1,h_2,h_3\in H$ and for any $y\in I$.
\end{definition}

Note that, in our case,
\begin{align*}
\nu_{h_1+h_2}(\sigma_{h_1+h_2}(y))+y&=\lambda_{s(h_1+h_2)}(s(h_1+h_2)^{-1}ys(h_1+h_2))+y\\
&=s(h_1+h_2)s(h_1+h_2)^{-1}ys(h_1+h_2)-s(h_1+h_2)+y\\
&=ys(h_1+h_2)-s(h_1+h_2)+y\\
&=y[s(h_1)+s(h_2)-\beta(h_1,h_2)]-s(h_1)-s(h_2)\\
&\quad +\beta(h_1,h_2)+y\\
&=ys(h_1)+y[s(h_2)-\beta(h_1,h_2)]-s(h_1)-s(h_2)\\
&\quad +\beta(h_1,h_2)\\
&=ys(h_1)-s(h_1)+ys(h_2)-y\beta(h_1,h_2)+y-s(h_2)\\
&\quad +\beta(h_1,h_2)\\
&=ys(h_1)-s(h_1)+ys(h_2)-s(h_2)\\
&=\nu_{h_1}(\sigma_{h_1}(y))+\nu_{h_2}(\sigma_{h_2}(y))
\end{align*}
using in one of the last steps that $y\beta(h_1,h_2)=y+\beta(h_1,h_2)$ because 
$I$ has trivial structure.

Since $B$ is a left brace, for any $h_1,h_2,h_3\in H$, $s(h_1)[s(h_2)+s(h_3)]+s(h_1)=s(h_1)s(h_2)+s(h_1)s(h_3)$. 
The left side of this equality can be rewritten using $\tau$ and $\beta$ as
\begin{align*}
s(h_1)[s(h_2)+s(h_3)]+s(h_1)&=s(h_1)[s(h_2+h_3)+\beta(h_2,h_3)]+s(h_1)\\
&=s(h_1)s(h_2+h_3)+s(h_1)\beta(h_2,h_3)\\
&=s(h_1(h_2+h_3))\tau(h_1,h_2+h_3)+s(h_1)\beta(h_2,h_3)\\
&=s(h_1(h_2+h_3))+\lambda_{s(h_1(h_2+h_3))}(\tau(h_1,h_2+h_3))\\
&\quad +s(h_1)+\lambda_{s(h_1)}(\beta(h_2,h_3))\\
&=s(h_1(h_2+h_3))+s(h_1)\\
&\quad +\nu_{h_1(h_2+h_3)}(\tau(h_1,h_2+h_3))+\nu_{h_1}(\beta(h_2,h_3))\\
&=s(h_1(h_2+h_3)+h_1)+\beta(h_1(h_2+h_3),h_1)\\
&\quad +\nu_{h_1(h_2+h_3)}(\tau(h_1,h_2+h_3))+\nu_{h_1}(\beta(h_2,h_3))\\
&=s(h_1h_2+h_1h_3)+\beta(h_1(h_2+h_3),h_1)\\
&\quad  +\nu_{h_1(h_2+h_3)}(\tau(h_1,h_2+h_3))+\nu_{h_1}(\beta(h_2,h_3))
\end{align*}
The other side becomes
\begin{align*}
s(h_1)s(h_2)+s(h_1)s(h_3)&=s(h_1h_2)\tau(h_1,h_2)+s(h_1h_3)\tau(h_1,h_3)\\
&=s(h_1h_2)+\lambda_{s(h_1h_2)}(\tau(h_1,h_2))+s(h_1h_3)\\
&\quad +\lambda_{s(h_1h_3)}(\tau(h_1,h_3))\\
&=s(h_1h_2)+s(h_1h_3)+\nu_{h_1h_2}(\tau(h_1,h_2))\\
&\quad +\nu_{h_1h_3}(\tau(h_1,h_3))\\
&=s(h_1h_2+h_1h_3)+\beta(h_1h_2,h_1h_3)+\nu_{h_1h_2}(\tau(h_1,h_2))\\
&\quad +\nu_{h_1h_3}(\tau(h_1,h_3)),
\end{align*}
and then we obtain 
\begin{align*}
\nu_{h_1(h_2+h_3)}(\tau(h_1,h_2+h_3))+\beta(h_1(h_2+h_3),h_1)+\nu_{h_1}(\beta(h_2,h_3))\\
=\nu_{h_1 h_2}(\tau(h_1,h_2))+\nu_{h_1h_3}(\tau(h_1,h_3))+\beta(h_1h_2,h_1h_3),
\end{align*}
as we wanted.

So we have seen that, given an exact sequence $0\to I\to B\to H\to 0$ of 
left braces with $I$ a trivial brace, we are able to 
construct an extension $(H,I,\sigma,\nu,\tau,\beta)$. 
In the next theorem we show that, conversely, that is all we need to construct all 
the braces $B$ containing $I$ as an ideal, and such that $B/I\cong H$.

\begin{theorem}\label{ThmExt}
Let $H$ be a left brace, and let $I$ be a trivial brace.
Assume that we have given an extension $(H,I,\sigma,\nu,\tau,\beta)$. 
Then, $H\times I$ is a left brace with sum 
$$
(h_1,y_1)+(h_2,y_2):=(h_1+h_2,y_1+y_2+\beta(h_1,h_2)),
$$
and product
$$
(h_1,y_1)\cdot(h_2,y_2):=(h_1\cdot h_2,~\nu_{h_1\cdot h_2}(\sigma_{h_2}(\nu^{-1}_{h_1}(y_1)))+\nu_{h_1}(y_2)+
\nu_{h_1\cdot h_2}(\tau(h_1,h_2))).
$$
Moreover, $\{0\}\times I$ is an ideal of $H\times I$ isomorphic to $I$ such that $(H\times I)/(\{1\}\times I)\cong H$. 

Conversely, if $B$ is a left brace, $I$ is an ideal of $B$ with trivial brace structure, and $H=B/I$, then 
$B$ is isomorphic to a brace of the form described above, constructed from $I$ and $H$.
\end{theorem}
\begin{proof}
First, we check that the given operations define a structure of left brace. 
The sum defines a structure of abelian group because $\beta$ is a $2$-cocycle of abelian groups, so it 
defines a extension of abelian groups. For the product, we have to check the three properties of a group: 
\begin{enumerate}[(i)]
\item The neutral element is $(0,0)$, since
$$
(h,y)\cdot (0,0)=(h\cdot 0,~\nu_{h\cdot 0}(\sigma_{0}(\nu^{-1}_h(y)))+\nu_h(0)+\nu_{h\cdot 0}(\tau(h,0)))=
(h,y),
$$
and 
$$
(0,0)\cdot (h,y)=(0\cdot h,~\nu_{0\cdot h}(\sigma_{h}(\nu^{-1}_0(0)))+\nu_0(y)+\nu_{0\cdot h}(\tau(0,h))=
(h,y),
$$
recalling that $\tau(0,h)=\tau(h,0)=0$, and that $0$ is both the additive and the multiplicative neutral element 
in a left brace.
\item The inverse of an element $(h,y)$ is $(h^{-1},~y'),$
with $y'=-\nu^{-1}_h(\tau(h,h^{-1}))-\nu^{-1}_h(\sigma^{-1}_h(\nu^{-1}_h(y)))$,
because
\begin{align*}
(h,y)\cdot (h^{-1},~y')&=(h\cdot h^{-1},~\nu_{hh^{-1}}(\tau(h,h^{-1}))+\nu_{hh^{-1}}[\sigma_{h^{-1}}(\nu^{-1}_h(y))]+
\nu_{h}(y'))\\
&=(0,~\tau(h,h^{-1})+\sigma^{-1}_h(\nu^{-1}_h(y))+\nu_h(y'))\\
&=(0,0),
\end{align*}
and 
\begin{align*}
(h^{-1},~y')\cdot (h,y)&=(h^{-1}h,~\nu_{h^{-1}h}(\tau(h^{-1},h))+\nu_{h^{-1}h}[\sigma_{h}(\nu^{-1}_{h^{-1}}(y'))]+
\nu_{h^{-1}}(y))\\
&=(0,~\tau(h^{-1},h)+\sigma_h(\nu_h(y'))+\nu^{-1}_h(y))\\
&=(0,~\tau(h^{-1},h)-\sigma_h(\tau(h,h^{-1})))\\
&=(0,0),
\end{align*}
where in the last step we make use of the $2$-cocycle property of $\tau$: $\tau$ satisfies 
$\sigma_{h_3}(\tau(h_1,h_2))+\tau(h_1\cdot h_2,h_3)=\tau(h_1,h_2\cdot h_3)+\tau(h_2,h_3)$; 
with $h_1=h_3=h$ and $h_2=h^{-1}$, we obtain $\sigma_h(\tau(h,h^{-1}))=\tau(h^{-1},h)$.

\item The associativity of $\cdot$ is proved by means of the cocycle property of $\tau$. On one side,
\begin{align*}
\lefteqn{(h_1,y_1)[(h_2,y_2)(h_3,y_3)]}\\
&=(h_1,y_1)
(h_2\cdot h_3,~\nu_{h_2\cdot h_3}[\sigma_{h_3}(\nu^{-1}_{h_2}(y_2))]+\nu_{h_2}(y_3)+
\nu_{h_2\cdot h_3}(\tau(h_2,h_3)))\\
&=(h_1h_2h_3,~\nu_{h_1h_2h_3}[\sigma_{h_2h_3}(\nu^{-1}_{h_1}(y_1))]+
\nu_{h_1h_2h_3}[\sigma_{h_3}(\nu^{-1}_{h_2}(y_2))]\\
&\quad +\nu_{h_1h_2}(y_3)+\nu_{h_1h_2h_3}(\tau(h_2,h_3))+
\nu_{h_1h_2h_3}(\tau(h_1,h_2h_3))~).
\end{align*}
On the other side,
\begin{align*}
\lefteqn{[(h_1,y_1)(h_2,y_2)](h_3,y_3)}\\
&=(h_1\cdot h_2,~\nu_{h_1\cdot h_2}(\sigma_{h_2}(\nu^{-1}_{h_1}(y_1)))+\nu_{h_1}(y_2)+
\nu_{h_1\cdot h_2}(\tau(h_1,h_2)))\cdot (h_3,y_3)\\
&=(h_1h_2h_3,~\nu_{h_1h_2h_3}[\sigma_{h_3}(\sigma_{h_2}(\nu^{-1}_{h_1}(y_1)))]+\nu_{h_1h_2h_3}[\sigma_{h_3}(\nu^{-1}_{h_2}(y_2))]\\
&\quad +\nu_{h_1h_2h_3}[\sigma_{h_3}(\tau(h_1,h_2))]+
\nu_{h_1h_2}(y_3)+\nu_{h_1h_2h_3}(\tau(h_1h_2,h_3))~),
\end{align*}
and the two sides are equal using $\sigma_{h_3}\sigma_{h_2}=\sigma_{h_2h_3}$ (i.e. $\sigma$ is a right action), and 
the the fact that $\tau$ is a $2$-cocycle with respect to $\sigma$. 
\end{enumerate}

Moreover, the two operations must be related by the brace property. This is proved using the property (\ref{propact})
that relates the actions $\nu$ and $\sigma$, and the property (\ref{propcocycle}) that relates the cocycles $\tau$ and $\beta$:
on one side,
\begin{align*}
\lefteqn{(h_1,y_1)[(h_2,y_2)+(h_3,y_3)]+(h_1,y_1)}\\
&=(h_1,y_1)(h_2+h_3,~y_2+y_3+\beta(h_2,h_3))+(h_1,y_1)\\
&=(h_1(h_2+h_3),~\nu_{h_1(h_2+h_3)}[\sigma_{h_2+h_3}(\nu^{-1}_{h_1}(y_1))]+\nu_{h_1}(y_2+y_3+\beta(h_2,h_3))\\
&\quad +\nu_{h_1(h_2+h_3)}(\tau(h_1,h_2+h_3))~)+(h_1,y_1)\\
&=(h_1(h_2+h_3)+h_1,~\nu_{h_1(h_2+h_3)}[\sigma_{h_2+h_3}(\nu^{-1}_{h_1}(y_1))]+\nu_{h_1}(y_2+y_3+\beta(h_2,h_3))\\
&\quad +\nu_{h_1(h_2+h_3)}(\tau(h_1,h_2+h_3))+y_1+\beta(h_1(h_2+h_3),h_1)~).\\
\end{align*}
On the other,
\begin{align*}
\lefteqn{(h_1,y_1)(h_2,y_2)+(h_1,y_1)(h_3,y_3)}\\
&=(h_1\cdot h_2,~\nu_{h_1\cdot h_2}(\sigma_{h_2}(\nu^{-1}_{h_1}(y_1)))+\nu_{h_1}(y_2)+
\nu_{h_1\cdot h_2}(\tau(h_1,h_2))~)\\
&\quad +(h_1\cdot h_3,~\nu_{h_1\cdot h_3}(\sigma_{h_3}(\nu^{-1}_{h_1}(y_1)))+\nu_{h_1}(y_3)+
\nu_{h_1\cdot h_3}(\tau(h_1,h_3))~)\\
&=(h_1h_2+h_1h_3,~\nu_{h_1\cdot h_2}(\sigma_{h_2}(\nu^{-1}_{h_1}(y_1)))+\nu_{h_1}(y_2)+
\nu_{h_1\cdot h_2}(\tau(h_1,h_2))\\
&\quad +\nu_{h_1\cdot h_3}(\sigma_{h_3}(\nu^{-1}_{h_1}(y_1)))+\nu_{h_1}(y_3)+
\nu_{h_1\cdot h_3}(\tau(h_1,h_3))+
\beta(h_1h_2,h_1h_3)~).
\end{align*}
The first components are equal because $H$ is a left brace. About the second components, the equality 
follows from (\ref{propact}) and (\ref{propcocycle}).

Note that the map $H\times I\to H$, $(h,y)\mapsto h$, is a surjective morphism of left braces, 
so its kernel $\{0\}\times I$ is an ideal, and $(H\times I)/(\{0\}\times I)\cong H$ as left braces. 

\vspace{0.3cm}
Finally, for the converse statement, using the notation from the beginning of this section, 
we consider the ideal $I$ of a left brace $B$, and the left brace $H=B/I$, 
and we define $\sigma_h(y):=s(h)^{-1}\cdot y\cdot s(h)$, 
$\nu_h(y):=\lambda_{s(h)}(y)$, $\tau(h_1,h_2):=s(h_1\cdot h_2)^{-1}\cdot s(h_1)\cdot s(h_2)$, and $\beta(h_1,h_2):=s(h_1)+s(h_2)-s(h_1+h_2).$
Then, we have checked before that all this defines an extension $(H,I,\sigma,\nu,\tau,\beta)$, and 
it is easy to check that the map $B\to H\times I$, defined by $s(h)+y\mapsto (h,y)$ for $h\in H$ and $y\in I$, 
is an isomorphism.
\end{proof}

These extensions with $I$ a trivial brace are not the most general possible, as we show in Remark \ref{semisimple}. 
But, as the next example shows, our results generalize the extensions by ideals contained in the socle, 
given in \cite[Corollary D]{NirBenDavid} and \cite[Theorem 2.1]{B}.

\begin{example}
Consider the left brace $B$ whose additive groups is $(\Z/(2))^3$, and whose multiplication is defined by
$$
\begin{pmatrix}
x_1\\
y_1\\
z_1\\
\end{pmatrix}\cdot
\begin{pmatrix}
x_2\\
y_2\\
z_2\\
\end{pmatrix}:=
\begin{pmatrix}
x_1\\
y_1\\
z_1\\
\end{pmatrix}+
\begin{pmatrix}
1& z_1& x_1+y_1+x_1z_1\\
0& 1& z_1+x_1+y_1z_1\\
0& 0& 1
\end{pmatrix} 
\begin{pmatrix}
x_2\\
y_2\\
z_2\\
\end{pmatrix},
$$
for any $(x_1,y_1,z_1)^t,(x_2,y_2,z_2)^t\in(\Z/(2))^3$. It appeared in \cite[Theorem 3.1]{B}. 
It is impossible to obtain it with \cite[Corollary D]{NirBenDavid}, because the socle 
of $B$ is trivial. Nevertheless, the set $I=\{(x,y,0)\in B : x,y\in\Z/(2)\}$ is an ideal 
of $B$ with trivial structure, so it is possible to obtain it as an extension of 
$B/I\cong \Z/(2)$ by $I$ using Theorem \ref{ThmExt}. Explicitly, one has to take 
$\tau(a,b)=(0,ab)$, $\beta(a,b)=(0,0)$, $\sigma_a(x,y)=(x,y+ax)$, and $\nu_a(x,y)=(x+ay,y)$, for any
$a,b,x,y,z\in\Z/(2)$.
\end{example}

From the point of view of extensions, the most basic braces are simple braces, because 
they are the starting pieces to build extensions. 
Until now, the only known examples of simple left braces were the cyclic groups $\Z/(p)$, where $p$ is a prime,
with product $a\cdot b:=a+b$ \cite[Corollary after Proposition 8]{Rump}. 
In the next sections, we will construct the first examples of non-trivial 
simple left braces. For this, we need to know how the Sylow subgroups of a finite left brace are related.

\section{Product of left ideals and Sylow $p$-subgroups}

Let $B$ be a left brace, and assume that $(B,+)=(G,+)\oplus (H,+)$, where $G$ and $H$ are left ideals. 
For $a,x\in G$ and $b,y\in H$, we can decompose the lambda maps as 
$\lambda_{a+b}(x+y)=\lambda_{a+b}(x)+\lambda_{a+b}(y)=
\lambda_{b}\lambda_{\lambda^{-1}_{b}(a)}(x)+\lambda_{a}\lambda_{\lambda^{-1}_{a}(b)}(y)$. 
Thus it is enough to know $\lambda_{x}(y)$, for $x\in G$ and $y\in H$, or for $x\in H$ and $y\in G$.
We write $\lambda_{b}(x)=:\alpha_b(x)$, $\lambda_{a}(y)=:\beta_a(y)$, 
$\lambda_{a}(x)=:\lambda^{(1)}_a(x)$, $\lambda_{b}(y)=:\lambda^{(2)}_b(y)$. Then, it is easy to see 
that these maps satisfy 
$$
\lambda_a^{(1)}\circ\alpha_b=\alpha_{\beta_a(b)}\circ\lambda^{(1)}_{\alpha^{-1}_{\beta_a(b)}(a)},
$$
and 
$$
\lambda_b^{(2)}\circ\beta_a=\beta_{\alpha_b(a)}\circ\lambda^{(2)}_{\beta^{-1}_{\alpha_b(a)}(b)},
$$
for $a\in G$ and $b\in H$.

\begin{definition}\label{matchedpair}
Let $G$ and $H$ be two left braces. Let
$\alpha: (H,\cdot)\to\aut(G,+)$ and $\beta: (G,\cdot)\to\aut(H,+)$ 
be group homomorphisms. We say that $(G,H,\alpha,\beta)$ is a \emph{matched pair of 
left braces} if $\alpha$ and $\beta$ satisfy
\begin{align}\label{property1}
\lambda_a^{(1)}\circ\alpha_b=\alpha_{\beta_a(b)}\circ\lambda^{(1)}_{\alpha^{-1}_{\beta_a(b)}(a)},
\end{align}
and 
\begin{align}\label{property2}
\lambda_b^{(2)}\circ\beta_a=\beta_{\alpha_b(a)}\circ\lambda^{(2)}_{\beta^{-1}_{\alpha_b(a)}(b)},
\end{align}
where $\alpha(b)=\alpha_b$ and $\beta(a)=\beta_a$, for all $a\in G$ and $b\in H$, and 
$\lambda^{(1)}$ and $\lambda^{(2)}$ are the lambda maps of the left braces $G$ and $H$ respectively.
\end{definition}

\begin{theorem}\label{product}
Let $(G,H,\alpha,\beta)$ be a matched pair of left braces. 
Then $G\times H$ is a left brace with sum
$$
(a,b)+(a',b'):=(a+a',b+b'),
$$
and lambda maps given by
$$
\lambda_{(a,b)}(a',b')=\left(\alpha_b\lambda^{(1)}_{\alpha^{-1}_b(a)}(a'),~\beta_a\lambda^{(2)}_{\beta^{-1}_a(b)}(b')\right).
$$ 
Moreover, $G\times\{0\}$ and $\{0\}\times H$ are two left ideals isomorphic to $G$ and $H$ as left braces, and such that 
$(G\times H,+)= (G\times\{0\})\oplus (\{0\}\times H)$.

Conversely, if $B$ is a left brace, and $I_1$ and $I_2$ are two left ideals of $B$ such that 
$(B,+)=(I_1,+)\oplus (I_2,+)$, then $B$ is isomorphic to one of the braces described before, 
constructed from $I_1$ and $I_2$. 
\end{theorem}
\begin{proof}
By Lemma \ref{defLambda}, we only have to check that 
$$\lambda_{(a,b)}\circ\lambda_{(x,y)}=\lambda_{(a,b)+\lambda_{(a,b)}(x,y)}.$$
By the definition of lambda maps, we can check that component by component separately, and then 
the last equality is equivalent to 
\begin{align}\label{equality}
\alpha_b\circ\lambda^{(1)}_{\alpha_b^{-1}(a)}\circ\alpha_{y}\circ\lambda^{(1)}_{\alpha_y^{-1}(x)}=
\alpha_{b+\beta_a\lambda^{(2)}_{\beta^{-1}_a(b)}(y)}
\circ
\lambda^{(1)}_{\alpha^{-1}_{b+\beta_a\lambda^{(2)}_{\beta^{-1}_a(b)}(y)}(a+\alpha_b\lambda^{(1)}_{\alpha^{-1}_b(a)}(x))},
\end{align}
$$
\beta_a\circ\lambda^{(2)}_{\beta_a^{-1}(b)}\circ\beta_x\circ\lambda^{(2)}_{\beta^{-1}_x(y)}=
\beta_{a+\alpha_b\lambda^{(1)}_{\alpha^{-1}_b(a)}(x)}
\circ
\lambda^{(2)}_{\beta^{-1}_{a+\alpha_b\lambda^{(1)}_{\alpha^{-1}_b(a)}(x)}(b+\beta_a\lambda^{(2)}_{\beta^{-1}_a(b)}(y))}.
$$

We are only going to check the first equality; the second one is analogous. Note that, by (\ref{property1}),
$a+\alpha_b\lambda^{(1)}_{\alpha^{-1}_b(a)}(x)=a+\lambda^{(1)}_a\alpha_{\beta^{-1}_a(b)}(x)=a\cdot\alpha_{\beta^{-1}_a(b)}(x)$, 
and by (\ref{property2}),
$b+\beta_a\lambda^{(2)}_{\beta^{-1}_a(b)}(y)=b+\lambda^{(2)}_b\beta_{\alpha^{-1}_b(a)}(y)=b\cdot\beta_{\alpha^{-1}_b(a)}(y)$ 
, so the right-hand side can 
be rewritten as
$$
\alpha_{b\cdot\beta_{\alpha^{-1}_b(a)}(y)}
\circ
\lambda^{(1)}_{\alpha^{-1}_{b\cdot\beta_{\alpha^{-1}_b(a)}(y)}(a\cdot\alpha_{\beta^{-1}_a(b)}(x))}=
\alpha_{b}\circ\alpha_{\beta_{\alpha^{-1}_b(a)}(y)}
\circ
\lambda^{(1)}_{\alpha^{-1}_{b\cdot\beta_{\alpha^{-1}_b(a)}(y)}(a\cdot\alpha_{\beta^{-1}_a(b)}(x))}.
$$
Now we analyse separately the term 
\begin{align*}
\alpha^{-1}_{b\cdot\beta_{\alpha^{-1}_b(a)}(y)}(a\cdot\alpha_{\beta^{-1}_a(b)}(x))&=
\alpha^{-1}_{b\cdot\beta_{\alpha^{-1}_b(a)}(y)}(a+\lambda^{(1)}_a\alpha_{\beta^{-1}_a(b)}(x))\\
&=\alpha^{-1}_{b\cdot\beta_{\alpha^{-1}_b(a)}(y)}(a)+
\alpha^{-1}_{b\cdot\beta_{\alpha^{-1}_b(a)}(y)}\lambda^{(1)}_a\alpha_{\beta^{-1}_a(b)}(x)\\
&=(\star)+(\star\star).
\end{align*}
The second summand is 
\begin{align*}
(\star\star)&=\alpha^{-1}_{b\cdot\beta_{\alpha^{-1}_b(a)}(y)}\lambda^{(1)}_a\alpha_{\beta^{-1}_a(b)}(x)\\
&=\lambda^{(1)}_{\alpha^{-1}_{b\cdot\beta_{\alpha^{-1}_b(a)}(y)}(a)}
\alpha^{-1}_{\beta^{-1}_a(b\cdot\beta_{\alpha^{-1}_b(a)}(y))}\alpha_{\beta^{-1}_a(b)}(x) \qquad\text{ (by (\ref{property1}))}\\
&=\lambda^{(1)}_{\alpha^{-1}_{b\cdot\beta_{\alpha^{-1}_b(a)}(y)}(a)}
\alpha^{-1}_y(x),
\end{align*}
where in the last equality we are using
\begin{align*}
\beta^{-1}_a(b\cdot\beta_{\alpha^{-1}_b(a)}(y))&=\beta^{-1}_a(b+\lambda^{(2)}_b\beta_{\alpha^{-1}_b(a)}(y))\\
&=\beta^{-1}_a(b)+\lambda^{(2)}_{\beta^{-1}_a(b)}\beta^{-1}_{\alpha^{-1}_b(a)}\beta_{\alpha^{-1}_b(a)}(y)
\qquad\text{ (by (\ref{property2}))}\\
&=\beta^{-1}_a(b)\cdot y
\end{align*}
So we get
\begin{align*}
\alpha^{-1}_{b\cdot\beta_{\alpha^{-1}_b(a)}(y)}(a\cdot\alpha_{\beta^{-1}_a(b)}(x))&=(\star)+(\star\star)\\
&=\alpha^{-1}_{b\cdot\beta_{\alpha^{-1}_b(a)}(y)}(a)+\lambda^{(1)}_{\alpha^{-1}_{b\cdot\beta_{\alpha^{-1}_b(a)}(y)}(a)}
\alpha^{-1}_y(x)\\
&=\alpha^{-1}_{\beta_{\alpha^{-1}_b(a)}(y)}\alpha^{-1}_b(a)\cdot \alpha_y^{-1}(x)\\
\end{align*}
and, summarizing, the right-hand side of (\ref{equality}) becomes equal to 
\begin{align}\label{RHSfinal}
\alpha_b\circ\alpha_{\beta_{\alpha^{-1}_b(a)}(y)}\circ
\lambda^{(1)}_{\alpha^{-1}_{\beta_{\alpha^{-1}_b(a)}(y)}\alpha^{-1}_b(a)}\circ\lambda^{(1)}_{\alpha_y^{-1}(x)}.
\end{align}

Now, if we apply the identity (\ref{property1}) to the middle terms of the left-hand side of (\ref{equality}), we get
$$
\alpha_b\circ\lambda^{(1)}_{\alpha_b^{-1}(a)}\circ\alpha_{y}\circ\lambda^{(1)}_{\alpha_y^{-1}(x)}=
\alpha_b\circ\alpha_{\beta_{\alpha^{-1}_b(a)}(y)}\circ
\lambda^{(1)}_{\alpha^{-1}_{\beta_{\alpha^{-1}_b(a)}(y)}\alpha^{-1}_b(a)}\circ\lambda^{(1)}_{\alpha_y^{-1}(x)},
$$ 
which is exactly (\ref{RHSfinal}), as we wanted.

Finally, for the converse statement, it suffices to review the argument that led us to Definition \ref{matchedpair}.
\end{proof}

\begin{definition}
Let $(G,H,\alpha,\beta)$ be a matched pair of left braces. Then, the left brace defined in Theorem \ref{product}
is called the matched product of $G$ and $H$. We denote it by $G\bowtie H$. 
\end{definition}

\begin{remark}
Observe that the additive group of a matched product $G\bowtie H$ of left braces is a direct 
product of the additive groups of $G$ and $H$. 
Note also that the multiplicative group of $G\bowtie H$ is in fact a matched product of 
the groups $(G,\cdot)$ and $(H,\cdot)$.
To see it, one can take the left action of $(H,\cdot)$ on the set $G$ as $^{b} a:=\alpha_b(a)$, 
and the right action of $(G,\cdot)$ on the set $H$ as $b^a:=\beta^{-1}_{\alpha_b(a)}(b)$, for any
$a\in G$ and $b\in H$ (see \cite[Definition IX.1.1]{Kassel}). 
\end{remark}

\begin{example}
Taking $\beta_a=\id$ for any $a\in G$ in the conditions of Theorem \ref{product}, we get the semidirect product of left braces, 
as is already defined in \cite{Rump4}. In particular, 
when $\alpha_b=\id$ and $\beta_a=\id$, we get the direct product of left braces.
\end{example}

\begin{example}\label{nB}
Let $B$ be a finite left brace. Then, for any positive integer $n$, 
the additive subgroup $nB=\left\{na: a\in B\right\}$
is closed by the lambda maps, because $\lambda_b\left(na\right)=n\lambda_b(a)$.
Hence $nB$ is a left ideal; in particular, it is also a subgroup of $(B,\cdot)$, because, if 
$a,a'\in nB$, then $a\cdot a'=a+\lambda_a(a')\in nB$.

For any decomposition of $m=\mid B\mid$ of the form $m=m'\cdot m''$ for two integers 
$m'$ and $m''$ such that $\gcd(m',m'')=1$, we obtain a left ideal $m''B$ of order 
$m'$, and a left ideal $m'B$ of order $m''$, and by Theorem \ref{product}, 
$B\cong (m'B)\bowtie (m''B)$. 

The last construction applies in particular when $m'=p^{\alpha}$ is a power of a prime
number, and $\gcd(p, m'')=1$. Then we obtain that any left brace is an iterated 
matched product of some of its Sylow subgroups. 
In particular, we can always follow the following technique to construct left braces over 
finite solvable groups: first, find brace structures over the Sylow subgroups, and second,
glue them together with a matched product (in case that it is possible). This is 
the technique that we follow in Section 6.
\end{example}

When $G$ and $H$ have a finite number of elements, and $\gcd(\ord{G},\ord{H})=1$, 
we can compute the socle of $G\bowtie H$.

\begin{proposition}
Let $(G,H,\alpha,\beta)$ be a matched pair such that $\ord{G},\ord{H}<\infty$ and 
$\gcd(\ord{G},\ord{H})=1$. Then, 
$\soc(G\bowtie H)=(\soc(G)\cap\ker(\beta))\times (\soc(H)\cap\ker(\alpha))$. 
\end{proposition}
\begin{proof}
With the hypothesis $\gcd(\ord{G},\ord{H})=1$, it is true that
$\im(\alpha)\cap\im(\lambda^{(1)})=0$, so 
$\alpha_b\lambda^{(1)}_{\alpha^{-1}_b(a)}=\id$ if and only if
$\alpha_b=\lambda^{(1)}_{\alpha^{-1}_b(a)}=\id$, or, equivalently,
if and only if $\alpha_b=\lambda^{(1)}_{a}=\id$. Doing the same with 
the second component, we have that $\beta_a\lambda^{(2)}_{\beta^{-1}_a(b)}=\id$ if and only if 
$\beta_a=\lambda^{(2)}_b=\id$.

Hence an element $(a,b)\in G\bowtie H$ belongs to the socle if and only if
$\alpha_b=\lambda^{(1)}_{a}=\beta_a=\lambda^{(2)}_b=\id$.
\end{proof}

If we want to find simple braces using matched pairs, in particular we want $\soc(G\bowtie H)$ to be $0$. 
By the last proposition, one easy way to obtain trivial socle in the matched pair is to 
have $\soc(G)=\soc(H)=0$. A family of braces of order a power of a prime and of trivial socle was 
given in \cite{Hegedus}. In the next section, we describe this family.

\section{A family of braces with trivial socle}\label{trivialsocle}

Consider the field $\mathbb{F}_p$ of $p$ elements, where $p$ is a prime.  
From now one we will work in the vector space $\mathbb{F}_p^{n+1}$, 
and we will use some linear and bilinear maps, so we establish the following notation. 

\paragraph{Notation:} We think the vector space $\mathbb{F}_p^{n+1}$ as 
$\mathbb{F}_p^{n}\times \mathbb{F}_p$; hence we
 denote any element of $\mathbb{F}_p^{n+1}$ by $(\vec{x},x_{n+1}),$
where $\vec{x}:=(x_1,\cdots,x_n)$ is an element of $\mathbb{F}_p^{n}$. Besides, 
we will denote linear or bilinear maps on $\mathbb{F}_p^n$ by lower-case letters, 
and their associated matrices with respect 
to the canonical basis, by capital letters (i.e. if $f$ is a linear endomorphism of $\mathbb{F}_p^n$,
then $f(\vec{x})=(F\cdot \vec{x}^t)^t$, where $F\in M_n(\mathbb{F}_p)$).

\paragraph{} Given a left brace $B$, recall that the socle of $B$ is defined as
$$
\soc(B)=\{a\in B:\lambda_a=\id\}=\{a\in B:a\cdot b=a+b\text{ for any } b\in B\}.
$$
Here we present a family of left braces with trivial socle, first appearing in \cite{CR,Hegedus}.
The concepts about quadratic forms that we use can be found for example in \cite{Pfister}

\begin{theorem}[\cite{Hegedus}, \cite{CR}]
Let $p$ be a prime number, and let $n\geq 1$ be an integer. Assume that we are given
 a quadratic form $Q$ on $\mathbb{F}_p^n$, and an element
$f$ of order $p$ of the orthogonal group of $Q$ (that is, an element $f\in \aut(\mathbb{F}_p^n)$ of order $p$
such that $Q(f(\vec{x}))=Q(\vec{x})$ for any $\vec{x}\in \mathbb{F}_p^n$). 
Then, the abelian group $\mathbb{F}_p^{n+1}$ with lambda maps 
$$
\lambda_{(\vec{x},x_{n+1})}(\vec{y},y_{n+1}):=(f^{q(\vec{x},x_{n+1})}(\vec{y}), ~y_{n+1}+b(\vec{x},f^{q(\vec{x},x_{n+1})}(\vec{y})))
$$
defines a structure of left brace, where $q(\vec{x},x_{n+1}):=x_{n+1}-Q(\vec{x})$, and $b$ is 
the associated bilinear form  $b(\vec{x},\vec{y}):=Q(\vec{x}+\vec{y})-Q(\vec{x})-Q(\vec{y})$.

Moreover, if $Q$ is non-degenerate, then the socle of this brace is equal to zero. 
\end{theorem}
\begin{proof}
First of all, observe that $\lambda_{(\vec{x},x_{n+1})}$ is well-defined since 
$q$ takes values in $\mathbb{F}_p$, and $f$ is of order $p$.
Next, in matrix notation, $\lambda$ can be written as
$$
\lambda_{(\vec{x},x_{n+1})}=
\left(\begin{array}{r|r}
F^{q(\vec{x},x_{n+1})}& 0\\
\hline
\vec{x}\cdot B\cdot F^{q(\vec{x},x_{n+1})}& 1\\
\end{array}\right).
$$

We shall prove that this defines a structure of left brace, checking the two properties required in Lemma \ref{defLambda}.
First, $\lambda_{(\vec{x},x_{n+1})}$ is a morphism of $(\mathbb{F}_p^n\times \mathbb{F}_p,+)$ 
by definition, and it is bijective because $f$ is invertible.
Second, we have to check 
$\lambda_{(\vec{x},x_{n+1})}\circ\lambda_{(\vec{y},y_{n+1})}=\lambda_{(\vec{x},x_{n+1})+\lambda_{(\vec{x},x_{n+1})}(\vec{y},y_{n+1})}$.
On one side,
\begin{eqnarray*}
\lambda_{(\vec{x},x_{n+1})}\circ\lambda_{(\vec{y},y_{n+1})}&=&
\left(\begin{array}{r|r}
F^{q(\vec{x},x_{n+1})}& 0\\
\hline
\vec{x}\cdot B\cdot F^{q(\vec{x},x_{n+1})}& 1\\
\end{array}\right)
\left(\begin{array}{r|r}
F^{q(\vec{y},y_{n+1})}& 0\\
\hline
\vec{y}\cdot B\cdot F^{q(\vec{y},y_{n+1})}& 1\\
\end{array}\right)\\
&=&
\left(\begin{array}{c|c}
F^{q(\vec{x},x_{n+1})}F^{q(\vec{y},y_{n+1})}& 0\\
\hline
\vec{x}\cdot B\cdot F^{q(\vec{x},x_{n+1})}F^{q(\vec{y},y_{n+1})}+\vec{y}\cdot B\cdot F^{q(\vec{y},y_{n+1})}& 1\\
\end{array}\right).
\end{eqnarray*}
On the other hand,
\begin{align*}
\lefteqn{\lambda_{(\vec{x},x_{n+1})+\lambda_{(\vec{x},x_{n+1})}(\vec{y},y_{n+1})}}\\
&=
\left(\begin{array}{c|c}
F^{q((\vec{x},x_{n+1})+\lambda_{(\vec{x},x_{n+1})}(\vec{y},y_{n+1}))}& 0\\
\hline
\left(\vec{x}+\vec{y}\cdot (F^t)^{q(\vec{x},x_{n+1})}\right) BF^{q((\vec{x},x_{n+1})+\lambda_{(\vec{x},x_{n+1})}(\vec{y},y_{n+1}))}& 1\\
\end{array}\right),
\end{align*}
where $F^t$ is the transpose of $F$. Thus it only remains to check that 
$q((\vec{x},x_{n+1})+\lambda_{(\vec{x},x_{n+1})}(\vec{y},y_{n+1}))=q(\vec{x},x_{n+1})+q(\vec{y},y_{n+1})$,
because then
\begin{align*}
\lefteqn{\vec{y} \cdot (F^t)^{q(\vec{x},x_{n+1})} BF^{q((\vec{x},x_{n+1})+\lambda_{(\vec{x},x_{n+1})}(\vec{y},y_{n+1}))}}\\
&=\vec{y}\cdot (F^t)^{q(\vec{x},x_{n+1})} BF^{q(\vec{x},x_{n+1})}F^{q(\vec{y},y_{n+1})}=\vec{y} BF^{q(\vec{y},y_{n+1})}.
\end{align*}
Indeed
\begin{align*}
\lefteqn{q((\vec{x},x_{n+1})+\lambda_{(\vec{x},x_{n+1})}(\vec{y},y_{n+1}))}\\
&=q(\vec{x}+f^{q(\vec{x},x_{n+1})}(\vec{y})~,~ x_{n+1}+y_{n+1}+b(\vec{x},f^{q(\vec{x},x_{n+1})}(\vec{y})))\\
&=x_{n+1}+y_{n+1}+b(\vec{x},f^{q(\vec{x},x_{n+1})}(\vec{y}))-Q(\vec{x}+f^{q(\vec{x},x_{n+1})}(\vec{y}))\\
&=x_{n+1}+y_{n+1}+b(\vec{x},f^{q(\vec{x},x_{n+1})}(\vec{y}))-Q(\vec{x})-Q(f^{q(\vec{x},x_{n+1})}(\vec{y}))\\
&\quad -b(\vec{x},f^{q(\vec{x},x_{n+1})}(\vec{y}))\\
&=x_{n+1}+y_{n+1}-Q(\vec{x})-Q(\vec{y})=q(\vec{x},x_{n+1})+q(\vec{y},y_{n+1}),
\end{align*}
and all this implies that we have a structure of left brace.

Finally, suppose that $Q$ is non-degenerate. Note that 
$(\vec{x},x_{n+1})$ is an element of the socle if and only if $f^{q(\vec{x},x_{n+1})}=\id$ and
$b(\vec{y},f^{q(\vec{x},x_{n+1})}(\vec{x}))=0$ for any $\vec{y}$. The combination of this two conditions gives
$b(\vec{y},\vec{x})=0$ for any $\vec{y}$. Since $Q$ is non-degenerate, $\vec{x}=0$, and $f^{q(\vec{x},x_{n+1})}=\id$
implies $0=q(\vec{x},x_{n+1})=x_{n+1}-Q(\vec{x})$ if $f$ has order $p$, which implies $x_{n+1}=0$. In conclusion, if the
quadratic form is non-degenerate, the socle is trivial.
\end{proof}

Recall that, if $p=2$, there exists a non-degenerate quadratic form on $\mathbb{F}_2^n$ if and only if
$n$ is even.


\section{Construction of simple braces with $\Z/(p_1)\times (\Z/(p_2))^{n+1}$ as additive group}
To find a finite simple left brace, we need to define matched products between finite $p$-groups, which 
will be the Sylow subgroups of the simple brace. Moreover, note the following proposition. Recall that, 
if $B$ is a finite left brace, then $(B,\cdot)$ is solvable by \cite[Theorem 2.15]{ESS}. Recall also that 
a Hall subgroup $H$ of a finite group $G$ is a subgroup such that $\gcd(\ord{H},~[G:H])=1.$

\begin{proposition}
Let $B$ be a finite left brace. Assume that one of the Hall subgroups $H$ of $(B,\cdot)$ is normal. Then, 
$H$ is an ideal of $B$.
\end{proposition}
\begin{proof}
Suppose that $\ord{H}=n$, and that $\ord{B}/\ord{H}=m$. Since $H$ is a Hall subgroup, 
$\gcd(n,m)=1$. From Example \ref{nB}, we know that $mB$ is a left ideal of $B$ of order $n$. 
Since $H$ is normal and $(B,\cdot)$ is solvable, $H$ is the unique Hall subgroup of $(B,\cdot)$ 
of order $n$, so $(mB,\cdot)=H$, and we are done because a left ideal with normal multiplicative 
group is an ideal. 
\end{proof}
\begin{corollary}
Let $B$ be a finite simple left brace. Then, any Hall subgroup of $(B,\cdot)$ is non-normal.
\end{corollary}

Thus the multiplicative group of a finite simple left brace can not be any solvable group. 
Therefore we would want $\alpha$ and $\beta$ to be both non-trivial, because it is necessary for the Sylow subgroups to be non-normal
in a simple brace. 

Instead of working in the general setting, we simplify things working in a concrete 
example: we work with left braces with additive group isomorphic to $\Z/(p_1)\times (\Z/(p_2))^{n+1}$, where 
$p_1$ and $p_2$ are different prime numbers. 
Take $H=\Z/(p_1)$ with its only possible brace structure, the trivial brace structure $h\cdot h':=h+h'$ 
(so $\lambda^{(1)}_h=\id$ for any $h\in H$). 
Take $K$ equal to one 
of the braces with trivial socle defined in Section \ref{trivialsocle}. So 
we have $(K,+)\cong (\Z/(p_2))^{n+1}$, and with $F$, $Q$, $b$, $q$ as in Section \ref{trivialsocle}, and then 
$$
\lambda^{(2)}_{(x_1,\dots,x_{n+1})}=
\begin{pmatrix}
F^{q(x_1,\dots,x_{n+1})}& 0\\
(x_1,\dots,x_n)\cdot B\cdot F^{q(x_1,\dots,x_{n+1})}& 1
\end{pmatrix}.
$$

To put together $H$ and $K$, we need two morphisms
$$
\alpha: (K,\cdot)\to \aut(\Z/(p_1))\cong(\Z/(p_1))^{*},
$$
$$
\beta: (H,\cdot)=\Z/(p_1)\to \aut((\Z/(p_2))^{n+1})\cong GL_{n+1}(\Z/(p_2)).
$$
For $\alpha$ to be non-trivial, we must assume $p_2\mid (p_1-1)$ (in particular, $p_1\neq 2$).
On the other hand, to define $\beta$, we only need an automorphism $m=\beta_{\overline{1}}$ of order $p_1$
(which is determined by a matrix $M\in GL_{n+1}(\Z/(p_2))$).
Note that the additive group of the final brace is isomorphic to 
$A=\Z/(p_1)\times (\Z/(p_2))^{n+1}$. Then, by Theorem \ref{product}, if the maps $\alpha$ and $M$
satisfy the properties
\begin{enumerate}
\item $\alpha$ is an action,
\item $\alpha_{m(\vec{x},x_{n+1})}=\alpha_{(\vec{x},x_{n+1})}$ (corresponding to (\ref{property1}) of Definition \ref{matchedpair}),
\item $\lambda^{(2)}_{(\vec{x},x_{n+1})}\circ m^a=m^{\alpha_{(\vec{x},x_{n+1})}(a)}\circ\lambda^{(2)}_{m^{-\alpha_{(\vec{x},x_{n+1})}(a)}(\vec{x},x_{n+1})}$
(corresponding to (\ref{property2}) of Definition \ref{matchedpair}; observe that
the case $a=1$ implies the other cases because $\alpha_{(\vec{x},x_{n+1})}$ is a morphism), 
\end{enumerate}
we have a structure of left brace over $\Z/(p_1)\times(\Z/(p_2))^{n+1}$ given by
\begin{align*}
\lambda_{(a,x_1,\dots,x_{n+1})}&:=
\left(\alpha_{(x_1,\dots,x_{n+1})}\circ\lambda^{(1)}_{\alpha^{-1}_{(x_1,\dots,x_{n+1})}(a)},~
\beta_a\circ\lambda^{(2)}_{\beta^{-1}_a(x_1,\dots,x_{n+1})}\right)\\
&=\left(\alpha_{(x_1,\dots,x_{n+1})}~,~
m^{a}\circ 
\lambda^{(2)}_{m^{-a}(x_1,\dots,x_{n+1})}
\right).
\end{align*}

There are different ways to choose $\alpha$, but the following one works specially well with the brace 
structure of $K$: let $\gamma\in (\Z/(p_1))^*$ be an element of order $p_2$
(recall that we are assuming $p_2\mid (p_1-1)$). Then, define $\alpha_{(x_1,\dots,x_{n+1})}:=\gamma^{q(x_1,\dots,x_{n+1})}$.
So the second condition $\alpha_{m(\vec{x},x_{n+1})}=\alpha_{(\vec{x},x_{n+1})}$ is now equivalent 
to $q(m(\vec{x},x_{n+1}))=q(\vec{x},x_{n+1})$.
Moreover, assume that 
$M=\begin{pmatrix}
C& 0\\
z& 1
\end{pmatrix}$ to simplify the argument. Then,
$$
\lambda_{(a,\vec{x},x_{n+1})}:=\left(\gamma^{q(\vec{x},x_{n+1})}~,~
\begin{pmatrix}
C& 0\\
z& 1
\end{pmatrix}^{a}
\begin{pmatrix}
F^{q(\vec{x},x_{n+1})}& 0\\
\vec{x}(C^t)^{-a}BF^{q(\vec{x},x_{n+1})}& 1
\end{pmatrix}
\right).
$$
With these assumptions, the conditions to be a left brace become
\begin{enumerate}[(i)]
\item $\begin{pmatrix}C& 0\\ z& 1\end{pmatrix}$ has order $p_1$; equivalently, $C$ has order $p_1$ and $z+zC+\cdots+zC^{p_1-1}=0$.
\item $\gamma$ is an element of order $p_2$ in $(\Z/(p_1))^*$.
\item $Q(c(\vec{x}))=Q(\vec{x})+z\cdot \vec{x}^t$: Since the second condition 
$$q(m(\vec{x},x_{n+1}))=q(\vec{x},x_{n+1})$$ with 
$M=\begin{pmatrix}
C& 0\\
z& 1
\end{pmatrix}$ becomes equivalent to 
$$x_{n+1}+z\vec{x}^t-Q(c(\vec{x}))=q(c(\vec{x}),x_{n+1}+z\vec{x}^t)=q(\vec{x},x_{n+1})=x_{n+1}-Q(\vec{x}).$$
Besides, note that the fact that $c$ has order $p_1$ and the fact that $Q(c(\vec{x}))=Q(\vec{x})+z\cdot \vec{x}^t$
together implies the condition $z+zC+\cdots+zC^{p_1-1}=0$, because $Q(\vec{x})=Q(c^{p_1}(\vec{x}))=
Q(\vec{x})+z\vec{x}+zC\vec{x}^t+\cdots+zC^{p_1-1}\vec{x}^t.$ So we do not need to assume this condition in (i).

\item $z=(z+zC+\cdots+zC^{\gamma-1})F$, and $F C=C^{\gamma}F$: The third condition
$$
\lambda^{(2)}_{(\vec{x},x_{n+1})}\circ m=m^{\alpha_{(\vec{x},x_{n+1})}(1)}\circ\lambda^{(2)}_{m^{-\alpha_{(\vec{x},x_{n+1})}(1)}(\vec{x},x_{n+1})}
$$ is now equivalent to 
\begin{align*}
\lefteqn{\begin{pmatrix}
F^{q(\vec{x},x_{n+1})}& 0\\
\vec{x}\cdot B\cdot F^{q(\vec{x},x_{n+1})}& 1
\end{pmatrix}
\begin{pmatrix}
C& 0\\
z& 1
\end{pmatrix}}\\
&=
\begin{pmatrix}
C& 0\\
z& 1
\end{pmatrix}^{\gamma^{q(\vec{x},x_{n+1})}}
\begin{pmatrix}
F^{q(\vec{x},x_{n+1})}& 0\\
\vec{x}\cdot (C^t)^{-\gamma^{q(\vec{x},x_{n+1})}}\cdot B\cdot F^{q(\vec{x},x_{n+1})}& 1
\end{pmatrix},
\end{align*}
which is equivalent, multiplying the matrices, to 
\begin{align*}
F^{q(\vec{x},x_{n+1})}C&=C^{\gamma^{q(\vec{x},x_{n+1})}}F^{q(\vec{x},x_{n+1})},\text{ and}\\
\vec{x}\cdot B\cdot F^{q(\vec{x},x_{n+1})}C+z&=(z+zC+\cdots +zC^{\gamma^{q(\vec{x},x_{n+1})}-1})F^{q(\vec{x},x_{n+1})}\\
&\quad +\vec{x}\cdot (C^t)^{-\gamma^{q(\vec{x},x_{n+1})}}\cdot B\cdot F^{q(\vec{x},x_{n+1})}.
\end{align*}
First, note that $FC=C^{\gamma}F$ implies $F^{x}C=C^{\gamma^{x}}F^{x}$ for any $x$. 
Next, $\vec{x}\cdot (C^t)^{-\gamma^{q(\vec{x},x_{n+1})}}\cdot B\cdot F^{q(\vec{x},x_{n+1})}=
\vec{x}\cdot B\cdot C^{\gamma^{q(\vec{x},x_{n+1})}}\cdot F^{q(\vec{x},x_{n+1})}= 
\vec{x}\cdot B\cdot F^{q(\vec{x},x_{n+1})}\cdot C$, because $C^tBC=B$ and using the first condition;
hence the second equality becomes 
$$
z=(z+zC+\cdots +zC^{\gamma^{q(\vec{x},x_{n+1})}-1})F^{q(\vec{x},x_{n+1})}.
$$
This condition can be derived from the other ones, because
$$
Q(cf^{-y}(\vec{x}))=Q(f^{-y}(\vec{x}))+z(f^{-y}(\vec{x}))^t=Q(\vec{x})+zF^{-y}\vec{x}^t
$$
and
\begin{align*}
Q(cf^{-y}(\vec{x}))&=Q(f^{-y}c^{\gamma^y}(\vec{x}))=Q(c^{\gamma^y}(\vec{x}))\\
&=Q(\vec{x})+(z+zC+\cdots+zC^{\gamma^y-1})\vec{x}^t.
\end{align*}
\end{enumerate}

We summarize all this information in the following theorem, where we also check when these 
left braces are simple. Recall the notation of lower-case letters for linear maps, and 
capital letters for their associated matrices.

\begin{theorem}\label{condSimple}
Let $p_1$ and $p_2$ be two prime numbers such that $p_2\mid p_1-1$. Let $n$ be a positive integer. 
Let $Q$ be a non-degenerate quadratic form over $(\Z/(p_2))^{n}$. Let $f$ be an element of order $p_2$ of 
the orthogonal group defined by $Q$. Assume that we are given an element $\gamma$ of $(\Z/(p_1))^*$, 
a matrix $C$ of $GL_{n}(\Z/(p_2))$, and $z\in (\Z/(p_2))^n$, satisfying the properties
\begin{enumerate}[(a)]
\item $\gamma$ has order $p_2$, and $C$ has order $p_1$,
\item $Q(c(\vec{x}))=Q(\vec{x})+z\cdot \vec{x}^t$, and
\item $F C=C^{\gamma}F$.
\end{enumerate} 
Then, there exists a brace with additive group isomorphic to $\Z/(p_1)\times(\Z/(p_2))^{n+1}$, given 
by the lambda maps
$$
\lambda_{(a,\vec{x},x_{n+1})}:=\left(\gamma^{q(\vec{x},x_{n+1})}~,~
\begin{pmatrix}
C& 0\\
z& 1
\end{pmatrix}^{a}
\begin{pmatrix}
F^{q(\vec{x},x_{n+1})}& 0\\
\vec{x}(C^t)^{-a}BF^{q(\vec{x},x_{n+1})}& 1
\end{pmatrix}
\right).
$$
In particular, for $p_2=2$, we must assume $n$ even, and, for $p_2\neq 2$, we always have $z=0$.

Moreover, this brace is simple if and only if 
$\im(c-\id)+\im(f-\id)=(\Z/(p_2))^n.$
In particular, a sufficient condition to be simple is that $c-\id$ is invertible.  
\end{theorem}
\begin{proof}
We have already checked that the first conditions guarantee a structure of left brace 
with those lambda maps. 
In particular, for $p_2\neq 2$, these conditions imply that $z$ is always zero because 
$b(c(\vec{x}),c(\vec{y}))=Q(c(\vec{x}+\vec{y}))-Q(c(\vec{x}))-Q(c(\vec{x}))=
Q(\vec{x}+\vec{y})+z(\vec{x}+\vec{y})^t-Q(\vec{x})-z\vec{x}^t-Q(\vec{y})-z\vec{y}^t=
b(\vec{x},\vec{y}),$ and this implies $Q(\vec{x})+z\vec{x}^t=Q(c(\vec{x}))=\frac{1}{2}b(c(\vec{x}),c(\vec{x}))=
\frac{1}{2}b(\vec{x},\vec{x})=Q(\vec{x})$, thus $z=0$.

Now to find the conditions for these braces to be simple, recall that an ideal $I$ of a left brace $B$ satisfies:
\begin{enumerate}[(1)]
\item $h_1+h_2\in I$ for all $h_1,h_2\in I$
\item $\lambda_g(h)\in I$ for all $g\in B$, $h\in I$
\item $ghg^{-1}\in I$ for all $g\in B$, $h\in I$. But $ghg^{-1}=\lambda_g(\lambda_h(g^{-1})-g^{-1}+h)$, and
a combination of (1) and (2) gives the equivalent condition:
\item[(3')] $(\lambda_h-\id)(g)\in I$ for all $h\in I$, $g\in B$
\end{enumerate}

So take $I\neq 0$ any ideal of $B$. To prove that $B$ is simple, we want to find conditions that ensures that $I=B$.
Since $I\neq 0$, there exists $(a,\vec{x},x_{n+1})\in I\setminus\{0\}$. If $a\neq 0$, then, since $p_1\neq p_2$,
$n(a,\vec{x},x_{n+1})=(1,0,\cdots,0)\in I$ for some $n$. If $a=0$, then $\lambda_{(b,\vec{y},y_{n+1})}(0,\vec{x},x_{n+1})\in I$
for any $(b,\vec{y},y_{n+1})\in B$, and we take $b=0$, $y_{n+1}=Q(y_1,\dots,y_n)$:
\begin{align*}
\lambda_{(0,\vec{y},Q(\vec{y}))}(0,\vec{x},x_{n+1})&=
\left(\left(
1,
\begin{pmatrix}
\Id& 0\\
\vec{y}\cdot B& 1
\end{pmatrix}
\right)
\begin{pmatrix}
0\\
\vec{x}^t\\
x_{n+1}
\end{pmatrix}\right)^t\\
&=
\begin{pmatrix}
0\\
\vec{x}^t\\
\vec{y}\cdot B\cdot \vec{x}^t+x_{n+1}
\end{pmatrix}^t\in I.
\end{align*}
Then
$(0,\vec{x},\vec{y}B\vec{x}^t+x_{n+1})-(0,\vec{x},x_{n+1})=(0,\vec{0},\vec{y}B\vec{x}^t)\in I$.
If $\vec{x}\neq 0$, then we choose $\vec{y}$ such that $\vec{y}\cdot B\cdot\vec{x}^t\neq 0$ 
(recall that $Q$ is non-degenerate),
so $(0,\vec{0},1)\in I$. But if $(0,\vec{0},1)\in I$, then
$$
(\lambda_{(0,\vec{0},1)}-\id)(1,\vec{0},0)=(\gamma-1,\vec{0},0),
$$
hence, since $\gamma\neq 1$, $(1,\vec{0},0)\in I$.
In conclusion, if $I\neq 0$, then always $(1,\vec{0},0)\in I$.

So our question reduces to: Does $(1,\vec{0},0)\in I$ imply $I=B$? If $(1,\vec{0},0)\in I$, then
$$
(\lambda_{(1,\vec{0},0)}-\id)(0,\vec{y},0)=
\left(\left(0,
\begin{pmatrix}
C-\Id& 0\\
z& 0
\end{pmatrix}
\right)
\begin{pmatrix}
0\\
\vec{y}^t\\
0
\end{pmatrix}\right)^t=
(0,(c-\id)(\vec{y}),z\vec{y}^t)\in I,
$$
choosing $\vec{y}$ such that $(c-\id)(\vec{y})\neq 0$ (recall that $C$ has order $p_1\neq 1$). 
Then, as in the last paragraph, we can conclude $(0,\vec{0},1)\in I$.
So we only have to check that $(0,\vec{x},0)\in I$ for any $\vec{x}$. With $(\lambda_{(a,\vec{0},0)}-\id)(b,\vec{y},y_{n+1})\in I$
and $(\lambda_{(0,\vec{0},x)}-\id)(b,\vec{y},y_{n+1})\in I$ for all $(b,\vec{y},y_{n+1})$, we see that
$(0,(c^a-\id)(\vec{y}),0)\in I$ and $(0,(f^x-\id)(\vec{y}),0)\in I$ for all $a$, $x$, and $\vec{y}$.
Condition (2) for ideals implies that the multiplication of any of these elements by $(1,F,1)$ or by $(1,C,1)$ is again an element of
$I$. Thus we obtain
$$
(0,(W_1(f,c)-W_2(f,c))(\vec{y}),0)\in I,
$$
for any $\vec{y}$, and every $W_1$, $W_2$ words in $f$ and $c$.
Note that Condition (3') does not add new elements to this set. So $(1,\vec{0},0)$, $(0,\vec{0},1)$, these elements, and all the possible
linear combinations of all of them (Condition (1)), are all the elements of $I$.

Thus the condition to be simple is that the elements 
$$\{(W_1(f,c)-W_2(f,c))(\vec{y}):\text{ for all }\vec{y}\text{, and for all words }W_1,W_2\}$$ generate
the vector space $(\Z/(p_2))^n$. But recall that $CF=FC^{\gamma^{-1}}$
Thus the condition to be simple is that the elements 
$$\{(f^\alpha c^\beta-f^\gamma c^\delta)(\vec{y}):\text{ for all } \alpha,\beta,\gamma,\delta,\vec{y}\}$$ generate
the vector space $(\Z/(p_2))^n$.  Or, equivalently, the elements of 
$$\{(f^\alpha c^\beta-\id)(\vec{y}):\text{ for all } \alpha,\beta,\vec{y}\}$$ generate $(\Z/(p_2))^n$.
Note that 
$(f^\alpha c^\beta-\id)(\vec{y})=(f^\alpha -\id)c^\beta(\vec{y})+(c^\beta-\id)(\vec{y})\in\im(f-\id)+\im(c-\id),$
since $f^\alpha-\id=(f-\id)(f^{\alpha-1}+\cdots +\id)$, and similarly for $c^\beta-\id$. Therefore, 
$$\left\langle (f^\alpha c^\beta-\id)(\vec{y}):\text{ for all } \alpha,\beta,\vec{y}\right\rangle_+=\im(f-\id)+\im(c-\id).$$
Summarizing, our brace is simple if and only if $\im(f-\id)+\im(c-\id)=(\Z/(p_2))^n.$
In particular, if $c-\id$ is invertible, $\im(c-\id)=(\Z/(p_2))^n,$ and the condition is fulfilled. This finishes the proof. 
\end{proof}

\section{Concrete examples}
In the last section, we have found sufficient conditions to have a simple brace. 
In this section, we show explicit examples of matrices and vectors satisfying the 
conditions of Theorem \ref{condSimple}, providing the first known examples of 
non-trivial simple left braces. 

\subsection{Examples of order $p_1p_2^{p_1}$}

In this section,
we construct an explicit example of simple left brace of order $p_1p_2^{p_1}$ 
for each pair of prime numbers $p_1$ and $p_2$ such that $p_2\mid (p_1-1)$. Observe that 
all the examples are of order $p_1^ap_2^b$, and that we also obtain 
examples of odd order, so the behaviour of simple left braces is very different to simple groups. 

Let $n=(p_1-1)/p_2\in\Z.$ Let $\alpha$ be a primitive element of $(\Z/(p_1))^{*}$, and take $\gamma:=\alpha^{n}$, 
which is an element of order $p_2$ in $(\Z/(p_1))^{*}$. Moreover, take
$$
C=\begin{pmatrix}
0& 0& \cdots& 0& -1\\
1& 0& \cdots& 0& -1\\
0& \ddots& \ddots& \vdots& \vdots\\
\vdots& \ddots& \ddots& 0& -1\\
0&\cdots& 0& 1& -1
\end{pmatrix}\in GL_{p_1-1}(\Z/(p_2)),
$$
which is a matrix with characteristic polynomial equal to $x^{p_1-1}+\cdots +x+1$, so it has order $p_1$. 
We look for a quadratic form $Q$ and a matrix $F$ to satisfy all the conditions 
of Theorem \ref{condSimple}. One quadratic form that works well is $$Q(\vec{x})=\sum_{1\leq i<j\leq p_1-1} x_ix_j.$$
$Q$ is a non-degenerate quadratic form because its associated bilinear form in matrix form is 
$$
B=\begin{pmatrix}
0& 1& \cdots &1\\
1& \ddots & \ddots & \vdots\\
\vdots& \ddots & \ddots & 1\\
1&\cdots & 1& 0
\end{pmatrix}\in GL_{p_1-1}(\Z/(p_2)),
$$
which has determinant $\det(B)=(-1)^{p_1}(p_1-2)\equiv 1\pmod{p_2}$, thus it is in\-ver\-ti\-ble.

Any element $(\Z/(p_1))^{*}$ can be written as $\alpha^i\gamma^j$ for a unique $i\in\{0,1,\dots,n-1\}$, and 
a unique $j\in\{0,1,\dots,p_2-1\}$. Hence there exists a bijective map $\pi:(\Z/(p_1))^{*}\to \{1,2,\dots,p_1-1\}\subset \Z$, 
where $\pi(\alpha^i\gamma^j)$ is equal to the unique representative of the class of $\alpha^i\gamma^j$
between $1$ and $p_1-1$. 
Now consider the permutation 
$$\theta:=\prod_{i=0}^{n-1} (\pi(\alpha^i),\pi(\alpha^i\gamma),\dots,\pi(\alpha^i\gamma^{p_2-1}))\in S_{p_1-1},$$
and take $F$ equal to the associated permutation matrix of $\theta$ in $GL_{p_1-1}(\Z/(p_2))$. Since 
$\theta$ is a product of disjoint cycles of length $p_2$, $F$ has order $p_2$. 
We have to check that $F$ is an element of the orthogonal group of $Q$. 
When $p_2\neq 2$, it is enough to check that $F^tBF=B$. Note that 
$$
B=
\begin{pmatrix}
1& \cdots & 1\\
\vdots &  &\vdots\\
1& \cdots & 1
\end{pmatrix}-\Id_{p_1-1}.
$$
Hence,
$$
F^tBF=
F^t\begin{pmatrix}
1& \cdots & 1\\
\vdots &  &\vdots\\
1& \cdots & 1
\end{pmatrix}F-F^t\cdot \Id_{p_1-1}\cdot F=
\begin{pmatrix}
1& \cdots & 1\\
\vdots &  &\vdots\\
1& \cdots & 1
\end{pmatrix}-\Id_{p_1-1}=B,
$$
because $\begin{pmatrix}
1& \cdots & 1\\
\vdots &  &\vdots\\
1& \cdots & 1
\end{pmatrix}$ does not change when we multiply it by a permutation matrix, and $F^{-1}=F^t$ for permutation matrices. 
On the other hand, when $p_2=2$, we have to check that $Q(f(\vec{x}))=Q(\vec{x})$ for any $\vec{x}\in (\Z/(2))^{p_1-1}$. 
In this case, $\gamma=-1$ and 
$$
F=\begin{pmatrix}
0&\cdots & 0 & 1\\
\vdots & \iddots & \iddots & 0\\
0&\iddots & \iddots& \vdots\\
1& 0& \cdots &0
\end{pmatrix},
$$
and therefore, 
\begin{align*}
Q(f(\vec{x}))&=\sum_{1\leq i<j\leq p_1-1} x_{p_1-i}x_{p_1-j}=\sum_{1\leq i<j\leq p_1-1} x_{p_1-j}x_{p_1-i}\\
&=\sum_{p_1-1\geq p_1-i>p_1-j\geq 1} x_{p_1-j}x_{p_1-i}=\sum_{1\leq i'<j'\leq p_1-1} x_{i'}x_{j'}=Q(\vec{x}).
\end{align*}

Now, we have to check the properties of $C$ with respect to $F$ and $Q$. First, we have to prove that 
$FCF^{-1}=C^{\pi(\gamma)}$. Let $\{e_1,e_2,\dots,e_{p_1-1}\}$ be the canonical basis of $(\Z/(p_2))^{p_1-1}$.
Take any $k\in\{1,\dots,p_1-1\}$. 
Recall that there exists an element $\alpha^i\gamma^j$ of $(\Z/(p_1))^{*}$ such that $\pi(\alpha^i\gamma^j)=k$. 
We can describe $F$ and $C$ as linear maps. $F$ is the linear map that acts over the canonical basis
as $f(e_{\pi(\alpha^i\gamma^j)})=e_{\pi(\alpha^i\gamma^{j+1})}$. 
$C$ can be described as the linear map that acts over the canonical basis as $c(e_{k})=e_{k+1}$ if $k\neq p_1-1$, and 
$c(e_{p_1-1})=-e_1-e_2- \dots -e_{p_1-1}$. Hence, we have to consider three different cases:
\begin{enumerate}[(1)]
\item If $k=\pi(\alpha^i\gamma^j)< p_1-\pi(\gamma),$ then, on one hand,
\begin{align*}
fcf^{-1}(e_k)&=fcf^{-1}(e_{\pi(\alpha^i\gamma^j)})=fc(e_{\pi(\alpha^i\gamma^{j-1})})=f(e_{\pi(\alpha^i\gamma^{j-1})+1})\\
&=f(e_{\pi(\alpha^i\gamma^{j-1}+1)})=e_{\pi(\alpha^i\gamma^{j}+\gamma)}=e_{k+\pi(\gamma)},
\end{align*}
and, on the other hand,
\begin{align*}
c^{\pi(\gamma)}(e_k)=c^{\pi(\gamma)-1}(e_{k+1})=(\dots)=e_{k+\pi(\gamma)}.
\end{align*}

\item If $k=\pi(\alpha^i\gamma^j)=p_1-\pi(\gamma)$ , we have $\alpha^i\gamma^{j-1}=\alpha^i\gamma^{j}\gamma^{-1}\equiv k
\gamma^{-1}\equiv -\gamma\gamma^{-1}=-1\pmod{p_1}$, so $\pi(\alpha^i\gamma^{j-1})=p_1-1$. 
Then, on one hand,
\begin{align*}
fcf^{-1}(e_k)&=fcf^{-1}(e_{\pi(\alpha^i\gamma^j)})=fc(e_{\pi(\alpha^i\gamma^{j-1})})=fc(e_{p_1-1})\\
&=f(-e_1-e_2-\dots-e_{p_1-1})=-e_1-e_2-\dots-e_{p_1-1},
\end{align*}
(using in the last equality that a permutation matrix fixes the vector $e_1+e_2+\dots+e_{p_1-1}$)
and, on the other hand,
\begin{align*}
c^{\pi(\gamma)}(e_k)&=c^{\pi(\gamma)-1}(e_{k+1})=(\dots)=c(e_{k+\pi(\gamma)-1})=c(e_{p_1-\pi(\gamma)+\pi(\gamma)-1})\\
&=c(e_{p_1-1})=-e_1-e_2-\dots -e_{p_1-1}.
\end{align*}
\item If $k=\pi(\alpha^i\gamma^j)>p_1-\pi(\gamma)$, since $k\leq p_1-1$, we have $k+\pi(\gamma)-p_1\leq\pi(\gamma)-1$. 
Then, on one hand,
\begin{align*}
fcf^{-1}(e_k)&=fcf^{-1}(e_{\pi(\alpha^i\gamma^j)})=fc(e_{\pi(\alpha^i\gamma^{j-1})})=f(e_{\pi(\alpha^i\gamma^{j-1})+1})\\
&=f(e_{\pi(\alpha^i\gamma^{j-1}+1)})=e_{\pi(\alpha^i\gamma^{j}+\gamma)}=e_{k+\pi(\gamma)-p_1}.
\end{align*}
On the other hand, if $x=p_1-1-k$, we obtain
\begin{align*}
c^{\pi(\gamma)}(e_k)&=c^{\pi(\gamma)-1}(e_{k+1})=(\dots)=c^{\pi(\gamma)-x}(e_{k+x})=c^{\pi(\gamma)-p_1+1+k}(e_{p_1-1})\\
&=c^{\pi(\gamma)-p_1+k}(-e_{1}-e_2-\dots -e_{p_1-1})=c^{\pi(\gamma)-p_1+k-1}(e_{1})\\
&=(\cdots)=e_{\pi(\gamma)-p_1+k}.
\end{align*}
\end{enumerate}
Since $fcf^{-1}(e_k)=c^{\pi(\gamma)}(e_k)$ for any $e_k$ in the canonical basis, we conclude that $FCF^{-1}=C^{\pi(\gamma)}$. 

Next, define the integer $m:=\binom{p_1-1}{2}=(p_1-2)\frac{p_1-1}{2}$, and take $$z:=(0,\dots,0,\overline{m})\in (\Z/(p_2))^{p_1-1}.$$ Then
\begin{align*}
Q(c(\vec{x}))&=\sum_{j=2}^{p_1-1} (-x_{p_1-1})(x_{j-1}-x_{p_1-1})+\sum_{1<i<j\leq p_1-1} (x_{i-1}-x_{p_1-1})(x_{j-1}-x_{p_1-1})\\
 &=\sum_{j=2}^{p_1-1} (-x_{p_1-1}x_{j-1}+x_{p_1-1}^2)+\sum_{1<i<j\leq p_1-1} (x_{i-1}x_{j-1}-x_{i-1}x_{p_1-1}\\
 &\quad -x_{p_1-1}x_{j-1}+x_{p_1-1}^2)\\
 &=\sum_{1\leq i<j\leq p_1-1} x_{p_1-1}^2-\sum_{j=2}^{p_1-1} x_{j-1}x_{p_1-1}+\sum_{1<i<j\leq p_1-1} x_{i-1}x_{j-1}\\
 &\quad -\sum_{1<i<j\leq p_1-1}( x_{i-1}x_{p_1-1}+x_{p_1-1}x_{j-1}).
\end{align*}
When $p_2=2$, using that $\sum_{1\leq i<j\leq p_1-1} 1=\binom{p_1-1}{2}= m$, that $x^2=x$ for any $x\in\Z/(2)$, and that 
\begin{align*}
\sum_{1<i<j\leq p_1-1}( x_{i-1}x_{p_1-1}+x_{p_1-1}x_{j-1})&=x_{p_1-1}\cdot\sum_{i=1}^{p_1-2}\sum_{j=i+1}^{p_1-2}(x_{i}+x_{j})\\
&=x_{p_1-1}\cdot (p_1-3)\cdot\sum_{k=1}^{p_1-2} x_k=0,
\end{align*} we get
\begin{align*}
Q(c(\vec{x}))&=\binom{p_1-1}{2} x_{p_1-1}^2+\sum_{i=2}^{p_1-1} x_{i-1}x_{p_1-1}+\sum_{1<i<j\leq p_1-1} x_{i-1}x_{j-1}\\
&=\overline{m} x_{p_1-1}+\sum_{1\leq i<j\leq p_1-1} x_i x_j=z\cdot\vec{x}^t+Q(\vec{x}).
\end{align*}
When $p_2\neq 2$, note that $m=(p_1-2)\frac{p_1-1}{2}\equiv 0\pmod{p_2}$ because $p_1-1$ is divisible by $p_2$. 
Hence, $z=0$, and the term $\sum_{1\leq i<j\leq p_1-1} x_{p_1-1}^2=\overline{m}x_{p_1-1}^2=0$ disappears. 
Moreover, 
\begin{align*}
-\sum_{1<i<j\leq p_1-1}( x_{i-1}x_{p_1-1}+x_{p_1-1}x_{j-1})&=-x_{p_1-1}\cdot (p_1-3)\cdot\sum_{k=1}^{p_1-2} x_k\\
&=2\cdot x_{p_1-1}\cdot\sum_{k=1}^{p_1-2} x_k,
\end{align*}
so we get 
\begin{align*}
Q(c(\vec{x}))&=-\sum_{j=2}^{p_1-1} x_{j-1}x_{p_1-1}+\sum_{1<i<j\leq p_1-1} x_{i-1}x_{j-1}\\
 &\quad +2x_{p_1-1}\cdot\sum_{k=1}^{p_1-2} x_k\\
 &=\sum_{j=2}^{p_1-1} x_{j-1}x_{p_1-1}+\sum_{1\leq i<j\leq p_1-2} x_{i}x_{j}\\
 &=\sum_{1\leq i<j\leq p_1-1} x_{i}x_{j}=Q(\vec{x}).
\end{align*}

Finally, $C$ has no eigenvector of eigenvalue $1$, since its characteristic polynomial is $x^{p_1-1}+\cdots+x+1$; hence 
$C-\Id$ is invertible. Since all the sufficient conditions of Theorem \ref{condSimple} are fullfilled, we have a simple left brace
of order $p_1\cdot p_2^{p_1}$ for every pair of prime numbers $p_1, p_2$ such that $p_2\mid p_1-1$.

\begin{example}
If $p_2=3$ and $p_1=7$, then $\alpha=3$ and $\gamma=2$, so the permutation is 
$(1,2,4)(3,6,5)$. The corresponding permutation matrix is
$$
F=\left(\begin{array}{rrrrrr}
0 & 0 & 0 & 1 & 0 & 0\\
1 & 0 & 0 & 0 & 0 & 0\\
0 & 0 & 0 & 0 & 1 & 0\\
0 & 1 & 0 & 0 & 0 & 0\\
0 & 0 & 0 & 0 & 0 & 1\\
0 & 0 & 1 & 0 & 0 & 0
\end{array}\right)\in GL_6(\Z/(3)).
$$
\end{example}

\begin{remark}\label{semisimple}
In particular, the last argument ensures us that 
there exists a simple left brace $A$ with additive group isomorphic to $\Z/(3)\times(\Z/(2))^3$.
The multiplicative group of this brace is isomorphic to $S_4$, with isomorphism given by
$S_4\to (A,\cdot)$, $(1,2,3)\mapsto (1,(0,0,0))$ and $(3,4)\mapsto (0,(1,1,0))$.

Consider the direct product of left braces $A\times A$. Then, the only ideals of this brace
are $0$, $A\times \{0\}$, $\{0\}\times A$, and $A\times A$: take $I\neq 0$ any non-zero ideal 
of $A\times A$. There exists $(b_1,b_2)\in I$, $(b_1,b_2)\neq (0,0)$, and suppose without loss of generality that 
$b_1\neq 0$. Since $Z(A,\cdot)=Z(S_4)=\{1\}$, there exists an element $y\in A$ such that 
$y^{-1}b_1y\neq b_1$. Then, $(b_1,b_2)^{-1}(y,1)^{-1}(b_1,b_2)(y,1)=(b_1^{-1}y^{-1}b_1y,1)=(b_1^{-1}y^{-1}b_1y,1)$
is an element of $I$, so $I\cap (A\times \{0\})$ is a non-zero ideal of $A\times \{0\}$. But 
$A$ is a simple left brace, so $I\cap (A\times \{0\})=A\times\{0\}$, implying $(A\times \{0\})\subseteq I,$
and this gives us the desired result.

Since $A\times A$ has no ideals with trivial structure, we cannot obtain it with 
the methods of Section \ref{extensions}.
\end{remark}

\subsection{New examples by recursion}

In this final section, we show that, given one simple brace of the form given by Theorem \ref{condSimple}, 
we can obtain an infinite number of simple braces 
by a recursion process. As in Theorem \ref{condSimple}, consider a brace with additive group equal to 
$\Z/(p_1)\times(\Z/(p_2))^{n+1}$, and lambda maps given by
$$
\lambda_{(a,\vec{x},x_{n+1})}=\left(\gamma^{q(\vec{x},x_{n+1})}~,~
\begin{pmatrix}
C& 0\\
z& 1
\end{pmatrix}^{a}
\begin{pmatrix}
F^{q(\vec{x},x_{n+1})}& 0\\
(\vec{x},x_n)BF^{q(\vec{x},x_{n+1})}& 1
\end{pmatrix}
\right),
$$
with all the needed conditions for $C$, $F$, etc. to define a simple brace structure. 
Then, for any $k\geq 1$, we can use Theorem \ref{condSimple} to define 
a simple brace structure over $\Z/(p_1)\times (\Z/(p_2))^{k\cdot n+1}$. 

We define the quadratic form 
$$
Q_k(x_1,x_2,\dots,x_{kn}):=\sum_{i=0}^{k-1} Q(x_{ni+1},x_{ni+2},\dots,x_{ni+n}).
$$
For this quadratic form, the associated bilinear form is given in matrix form by
$$
B_k:=
\begin{pmatrix}
B &  & \\
 & \ddots & \\
 & & B
\end{pmatrix}.
$$
Then, using $F$, we can define an element $F_k$ of the orthogonal group defined by $Q_k$
as
$$
F_k:=\begin{pmatrix}
F &  & \\
 & \ddots & \\
 & & F
\end{pmatrix}. 
$$
Moreover, we can define a matrix $C_k$ and a vector $z_k$ in such a way that, together with $F_k$ and $Q_k$, they 
define a structure of simple left brace:
$$
C_k:=\begin{pmatrix}
C &  & \\
 & \ddots & \\
 & & C
\end{pmatrix},
$$
$$
z_k:=(z,~\stackrel{(k}{\dots}~,z).
$$

\section*{Acknowledgments}
I would like to thank Ferran Ced\'o for stimulating discussion about left braces, and for carefully reading the 
first version of this manuscript. 

Research partially supported by the grants DGI MICIIN
MTM2011-28992-C02-01, and MINECO MTM2014-53644-P.

\bibliographystyle{amsplain}
\bibliography{testbib}

\vspace{30pt}
 \noindent \begin{tabular}{llllllll}
 D. Bachiller  \\
 Departament de Matem\`atiques  \\
 Universitat Aut\`onoma de Barcelona   \\
08193 Bellaterra (Barcelona), Spain  \\
dbachiller@mat.uab.cat
\end{tabular}

\end{document}